\documentclass [12pt]{amsart}

\usepackage{amssymb,amsxtra,amsfonts}
\usepackage{epsf}              %
\usepackage{epsfig}             %
\usepackage{graphics}

\usepackage[colorlinks]{hyperref}

\usepackage{verbatim}  %for \comment
\usepackage{bbding}   % for \bigoasterisk----now use made-up version 

\newenvironment{ppb}[1]
{\ \!\!\!\!\!\!\!\!\!\!\!\!\!\!\!\!\!\!\!\!\!\!\!\!\!\!\!\!\!\!\!\!\!\!\!\!\!\!\!\! {\bf PPB-----------------------------------------------------------------------------PPB}\newline \tiny {#1}
\  \newline\normalsize\phantom{f}\!\!\!\!\!\!\!\!\!\!\!\!\!\!\!\!\!\!\!\!\!\!\!\!\!\!\!\!\!\!\!\!\!\!\!\!\!\!\!\! {\bf 
PPB-----------------------------------------------------------------------------PPB}\newline}{}

\long\def\pb #1*/{}

\def\reE@DeclareMathSymbol#1#2#3#4{%
    \let#1=\undefined
    \DeclareMathSymbol{#1}{#2}{#3}{#4}}
\DeclareSymbolFont{symbolsC}{U}{txsyc}{m}{n}
\SetSymbolFont{symbolsC}{bold}{U}{txsyc}{bx}{n}
\DeclareFontSubstitution{U}{txsyc}{m}{n}
\reE@DeclareMathSymbol{\strictiff}{\mathrel}{symbolsC}{76}

\newcommand\beq{\begin{equation}}
\newcommand\eeq{\end{equation}}
\newcommand\bal{\begin{align*}}
\newcommand\eal{\end{align*}}   %why does this not work??
\newcommand\bmx{\left(\begin{matrix}}
\newcommand\emx{\end{matrix}\right)}
\newcommand\bsmx{\left(\begin{smallmatrix}}
\newcommand\esmx{\end{smallmatrix}\right)}

\newcommand{\bSi}{{\bf \Si}}

\newcommand{\onto}{\twoheadrightarrow}
\newcommand{\into}{\hookrightarrow}
\newcommand{\spq}{/\!\!/}

\providecommand{\spqa}[1]{\underset{#1}{/\!\!/}}

\newcommand{\st}{\ \bigl\vert\ }

\providecommand{\mult}{\text{\rm mult}}

\def\part#1{\frac{\partial\phantom{q}}{\partial#1}}

\newcommand {\flb}{\lbrack\!\lbrack}
\newcommand {\frb}{\rbrack\!\rbrack}
\newcommand {\flp}{(\!(}
\newcommand {\frp}{)\!)}

\newcommand{\glue}[1]{\underset{#1}{\strictiff}}
\newcommand{\fus}{\circledast}

\newcommand{\MB}{\mathcal{M}_{\text{\rm B}}}
\newcommand{\wMB}{\wt{\mathcal{M}}_{\text{\rm B}}}

\newcommand{\HH}{\text{\rm H}}

\newcommand{\Lie}{{\mathop{\rm Lie}}}
\DeclareMathOperator{\ISto}{{\IS}to} 
\DeclareMathOperator{\isto}{sto}     %could do {\Is}to
\DeclareMathOperator{\sto}{sto}     %no small blackboard bold

\DeclareMathOperator{\Sect}{\mathop{\rm Sect}}

\newcommand{\papk}[3]{\,_{#1}^{\phantom{#3}}\cA_{#2}^{#3}     }
\newcommand{\Ad}{{\mathop{\rm Ad}}}
\newcommand{\ad}{{\mathop{\rm ad}}}

\newcommand{\ram}{\mathop{\rm ram}}
\newcommand{\lev}{\mathop{\rm lev}}

\newcommand{\Map}{{\mathop{\rm Map}}}
\newcommand{\Prod}{\prod}

\DeclareMathOperator{\Perm}{Perm}   
   
\DeclareMathOperator{\Iso}{Iso}   
\DeclareMathOperator{\GrIso}{GrIso}   
\DeclareMathOperator{\GrAut}{GrAut}

\DeclareMathOperator{\Hom}{Hom}         % this looks to be the correct way to do this (put a * ``...otar*{'' if want things underneath
\DeclareMathOperator{\Aut}{\mathop{\rm Aut}}
\DeclareMathOperator{\THom}{THom}

\newcommand{\GL}{{\mathop{\rm GL}}}

\DeclareMathOperator{\mon}{mon}	%monodromy

\newcommand{\Ker}{\mathop{\rm Ker}}

\DeclareMathOperator{\End}{End}

\newcommand{\hk}{{hyperk\"ahler }}   %useful word
\newcommand{\bD}{{\bf D}}

\newcommand{\bH}{{\bf H}}

\newcommand{\bQ}{{\bf Q}}
\newcommand{\bs}{{\bf S}}
\newcommand{\bS}{{\bf S}}

\DeclareSymbolFont{bbold}{U}{bbold}{m}{n}
\DeclareSymbolFontAlphabet{\mathbbold}{bbold}

\newcommand{\IA}{\mathbb{A}}

\newcommand{\IC}{\mathbb{C}}
\newcommand{\ID}{\mathbb{D}}

\newcommand{\IG}{\mathbb{G}}
\newcommand{\IH}{\mathbb{H}}

\newcommand{\IL}{\mathbb{L}}

\newcommand{\IN}{\mathbb{N}}

\newcommand{\IQ}{\mathbb{Q}}                           
                           
\newcommand{\IS}{\mathbb{S}}

\newcommand{\IT}{\mathbb{T}}

\newcommand{\IZ}{\mathbb{Z}}

\newcommand{\cA}{\mathcal{A}}

\newcommand{\cC}{\mathcal{C}}

\newcommand{\cF}{\mathcal{F}}
\newcommand{\cG}{\mathcal{G}}
\newcommand{\cH}{\mathcal{H}}
\newcommand{\ch}{\eta}     %{\hslash}    %\hbar, hslash..mathcal fails
\newcommand{\cI}{\mathcal{I}}

\newcommand{\cO}{\mathcal{O}}

\newcommand{\cP}{\mathcal{P}}

\newcommand{\cT}{\mathcal{T}}

\newcommand{\g}{       \mathfrak{g}     }

\newcommand{\lt}{\mathfrak{t}}
\newcommand{\lh}{\mathfrak{h}}

\newcommand{\ls}{\mathfrak{s}}

\newcommand{\wt}{\widetilde}

\newcommand{\wh}{\widehat}

\newcommand{\al}{\alpha}

\newcommand{\be}{\beta}
\newcommand{\ga}{\gamma}

\newcommand{\De}{\Delta}

\newcommand{\Ga}{\Gamma}

\newcommand{\la}{\lambda}

\newcommand{\si}{\sigma}

\newcommand{\Si}{\Sigma}
\renewcommand{\th}{\theta}

\newcommand{\ze}{\zeta}

\renewcommand{\bar}{\overline}

\makeatletter
 \newlength{\typesize}
\makeatother

\newlength{\vvoff}
\newlength{\hhoff}

\def\mapright#1{\smash{
        \mathop{\longrightarrow}\limits^{#1}}}

\def\mapdown#1{\Big\downarrow
        \rlap{$\vcenter{\hbox{$\scriptstyle#1$}}$}}

\def\underset#1#2{\ \smash{\mathop{ #2 }\limits_{#1}}\ }

\newcommand{\pf}{\begin{bpf}}

\newcommand{\pfms}{\begin{bpfms}}
\newcommand{\epf}{\end{bpf}\hfill$\square$\\}           % end proof
\newcommand{\epfms}{\end{bpfms}\hfill$\square$\\}       % end proof

\newcommand{\idea}{\begin{bidea}}

\newcommand{\eidea}{\end{bidea}\hfill$\square$\\}           % end proof

\newcommand{\sk}{\begin{bsk}}    %type: \sk ..... \esk

\newcommand{\esk}{\end{bsk}\hfill$\square$\\}           % end sketch
\newcommand{\sketch}{\begin{bsketch}}%type: \sketch ..... \esketch

\newcommand{\esketch}{\end{bsketch}\hfill$\square$\\}

\newtheorem {hypo}{\bf\hspace{-\parindent}Hypothesis}
\newtheorem {thm}[hypo]{Theorem}   %remove [hypo] to number separately
\newtheorem {prop}[hypo]{Proposition}%[section]
\newtheorem {cor}[hypo]{Corollary}%[section]
\newtheorem {lem}[hypo]{Lemma}%[section]

\newtheorem {defn}[hypo]{Definition}%[section]
\theoremstyle{remark}\newtheorem{rmk}[hypo]{Remark}

\newenvironment{exercise}{{\bf\hspace{-\parindent}Exercise.}}{}
\newenvironment{newq}[1]{\newpage}{}
\renewenvironment{newq}[1]{}{}

\begin{document}

\title{Twisted wild character varieties}
\author{Philip Boalch and Daisuke Yamakawa}%

\begin{abstract}
We will construct twisted versions of the wild character varieties. 
\end{abstract}

\maketitle

\section{Introduction}

This article extends the algebraic construction of the wild character varieties \cite{saqh, fission, gbs} 
to the case of 
``twisted'' Stokes local systems.
There are two, closely related, types of twist that appear, both already mentioned in \cite{gbs}.
Firstly we will consider Stokes data for connections with  twisted formal normal forms (often called the
``ramified case'').
A simple example where this type of twist occurs is the Airy equation, which was studied in Stokes' original article on the Stokes phenomenon \cite{stokes1857}. 
Here the ``twist'' can be understood visually from the following diagram, drawn by Stokes:

\begin{figure}[h]
	\resizebox{5cm}{!}{\includegraphics[width=5cm]{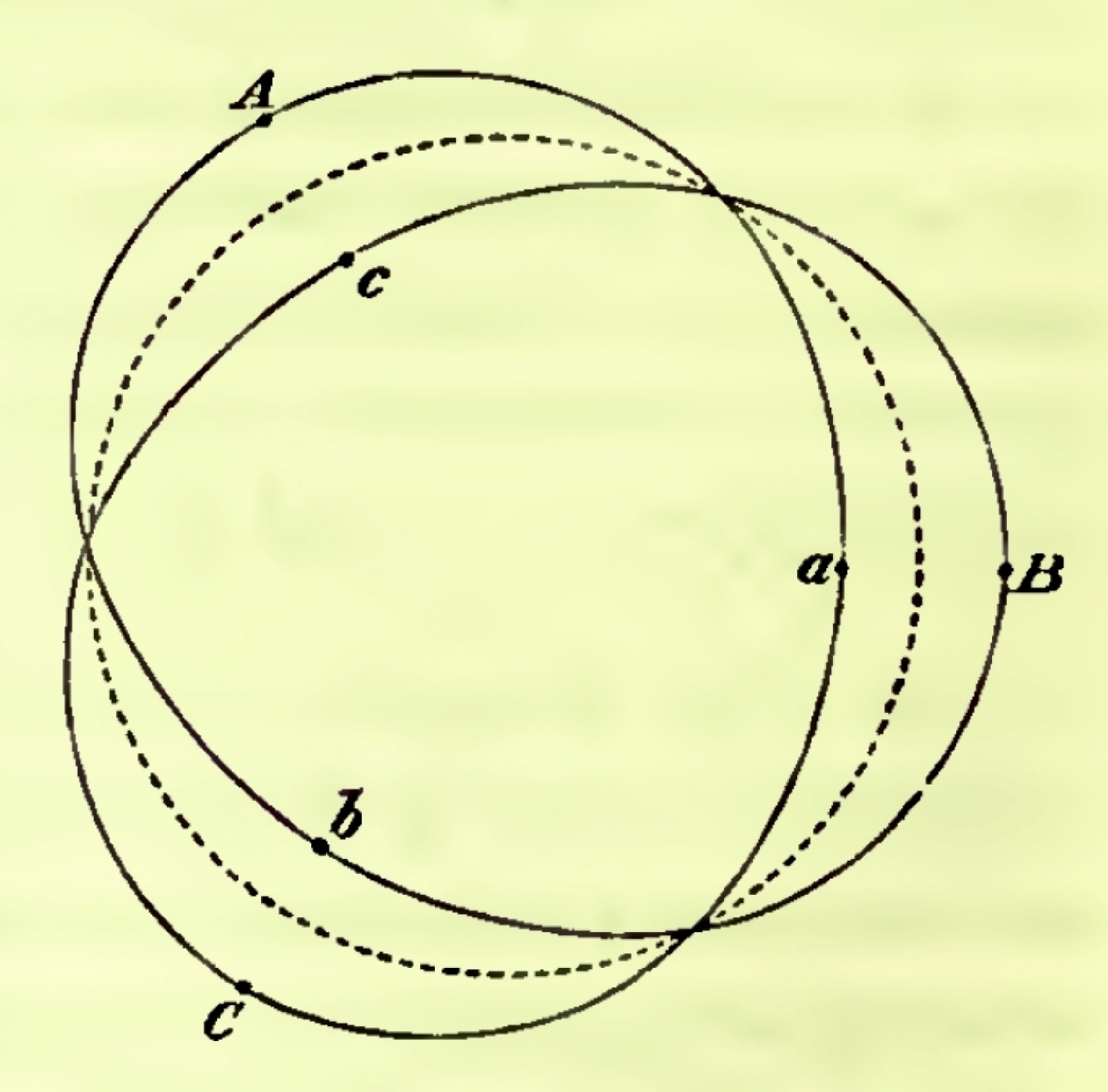}}
	\caption{The Stokes diagram of the Airy equation, from \cite{stokes1857} p.116.}\label{fig: stokesdiag} 
\end{figure}

Such twists can be recognised by the appearance of fractional powers of a local coordinate in the exponential factors of formal solutions, 
such as $\exp(\pm 2 x^{3/2})$  
in  %
the Airy equation $u''=9xu$ considered by Stokes (see \cite{stokes1857} eqn. (20)). 
The Stokes diagram represents the monodromy and the growth/decay of this exponential factor.

Secondly we will also consider local systems which are twisted in the interior of the curve---namely they are torsors for a local system of groups, rather than for a constant group.
Such examples may appear exotic, as for twisted loop groups, but in fact occur as the spaces of Stokes data used to classify 2d topological quantum field theories and the counting of BPS states \cite{cec-vafa-nequals2classn, Dub95long}.
These examples are partly responsible for the recent resurgence of interest in the Stokes phenomenon (cf. e.g. the introduction  to \cite{smid, smapg}).
They led to the discovery that the Poisson manifolds underlying the Drinfeld--Jimbo quantum groups are  simple examples of wild character varieties \cite{smapg}
and the subsequent realisation \cite{bafi} that the braiding of Stokes data
(the wild mapping class group) underlies Lusztig's symmetries (also known as Soibelman, Kirillov--Reshetikhin's quantum Weyl group).

Another motivation is that the wild character varieties
give the simplest description of the differentiable manifolds underlying a large class of complete \hk manifolds, the wild Hitchin spaces \cite{Hit-sde, wnabh}.
These are moduli spaces of solutions of Hitchin's self-duality equations; in one complex structure they are spaces of meromorphic Higgs bundles (and so algebraically completely integrable systems \cite{Hit-sbis, Bot, Mar}), 
and in another they are 
spaces of meromorphic connections 
(and so are spaces of initial conditions for the isomonodromy systems---see \cite{smid}).
By the irregular Riemann--Hilbert correspondence such spaces of connections become wild character varieties.
Thus there are three parallel classification problems that one could consider:

1) algebraic integrable systems,

2) isomonodromy systems,

3) wild character varieties.

If restricted to algebraic integrable systems that are isomorphic to a meromorphic Hitchin system, e.g. basically anything admitting a rational Lax pair, then these classifications are abstractly equivalent (via the wild versions of the non-abelian Hodge and Riemann--Hilbert correspondences on curves).
One point here is that many of the simplest 
algebraic integrable systems, such as the Mumford system (which is closely related to the
KdV hierarchy), are diffeomorphic to twisted wild character varieties in this way, as is the whole Painlev\'e 1 hierarchy of isomonodromy systems.
In subsequent work the present results will be used to extend the theory of Dynkin diagrams for the above classification problems (developed in \cite{rsode, slims, hi-ya-nslcase}) to the twisted case. They also yield a twisted version of the 
(coloured) multiplicative quiver varieties \cite{CB-Shaw, yamakawa-mpa, cmqv}.

A related motivation comes from mirror symmetry and Langlands duality. Namely via the above correspondences the wild character varieties inherit special Lagrangian torus fibrations 
(by \hk rotation), 
so can be studied in the SYZ framework of mirror symmetry.
By considering as-general-as-possible complex character varieties, we hope to get a complete picture 
which is closed under passing to the dual torus fibration. This duality is known to be related to Langlands duality in some cases \cite{ha-thad.mirror, DonPan12}, 
but  apparently suffers from 
``missing orbits'' in certain untwisted cases \cite{gukov-witten-rigid}.

\subsection{Summary of results}

We will first quickly review some known results.  
The quasi-Hamiltonian approach \cite{AMM} reveals the 
natural differential-geometric structure on moduli spaces of framed local systems on Riemann surfaces, that yields the natural Poisson/symplectic structures on the moduli spaces of local systems themselves, upon forgetting the framing.
Whilst first developed only for compact groups the theory may be complexified and then 
the complex character varieties may be constructed as multiplicative symplectic quotients of a smooth affine variety by a reductive group.
Suppose $\wh\Si$ is a compact Riemann surface with non-empty boundary $\partial$ and $G$ is a connected complex reductive group 
(with an invariant nondenegerate symmetric bilinear form on its Lie algebra).
Let $\be\subset \partial$ be a set of basepoints with exactly one element in each boundary circle $\partial_i\subset\partial$.
Then the space of $G$-local systems framed at $\be$ is isomorphic to the representation space
$$\Hom(\Pi,G)$$
where $\Pi=\Pi_1(\wh \Si,\be)$ is the fundamental groupoid.
The main result of \cite{AMM} basically says that this is a 
quasi-Hamiltonian space for the group $G^\be=\Map(\be,G)$.
This is a multiplicative analogue of the notion of a Hamiltonian space: there is an action of $G^\be$, an invariant two-form and a moment map taking values in the group $G^\be$. Here the action corresponds to changing the framings and the moment map is given by taking the local monodromies around the boundary circles $\rho \mapsto \{\rho(\partial_i)\}$.
An immediate corollary is that the character variety (the affine quotient)
$$\MB = \Hom(\Pi,G)/G^\be\cong
\Hom(\pi_1(\wh \Si),G)/G $$
inherits the structure of algebraic Poisson variety, with symplectic leaves obtained by fixing the conjugacy classes of the local monodromies.
Thus  the character varieties arise from  the smooth affine varieties 
$\Hom(\Pi,G)$ and their quasi-Hamiltonian structure.
In particular this gives a clean algebraic way to recover the Atiyah--Bott symplectic structure.

In \cite{saqh, fission, gbs} a similar picture was established for (untwisted) Stokes local systems, 
yielding many new examples of quasi-Hamiltonian spaces and 
an algebraic approach to the irregular extension \cite{smid} of the Atiyah--Bott symplectic structure.
Here the main new features are that the structure groups are naturally reduced to reductive subgroups 
$$H_i\subset G$$
in a halo around each boundary circle, and some extra, tangential, punctures are added to account for the Stokes multipliers.
The resulting space of Stokes representations
$$\Hom_\IS(\Pi,G)\subset \Hom(\Pi,G)$$
is again a smooth affine variety and Thm 1.1 of \cite{gbs} says that it is a quasi-Hamiltonian $\bH$-space where $\bH= \Prod H_i\subset G^\be$.
(Here $\Pi$ is defined after removing the tangential punctures from $\wh \Si$.)
This implies the wild character varieties $\Hom_\IS(\Pi,G)/\bH$ are Poisson with symplectic leaves obtained by fixing the conjugacy classes of the formal monodromies (in the groups $H_i$).
This work is surveyed in \cite{p12}. 
The possibility to break $G$ to $H_i$ is called fission, and leads to a bestiary of new symplectic varieties, beyond those of concern here \cite{fission}, \cite{gbs} \S3.2.

On the other hand, returning to the tame case, physically it doesn't seem quite natural to restrict to $G$-local systems (requiring some omnipotent being to transport a fixed copy of $G$ from one piece of the surface to another). Rather it is possible that as one moves around the surface the group $G$ becomes twisted by an automorphism.
Thus suppose $\cG\to \wh \Si$ is a local system of groups (with each fibre isomorphic to $G$).
It is determined by a morphism $f:\Pi\to \Aut(G)$.
Then we can consider $\cG$-local systems on $\wh \Si$.
Now however the monodromy around a loop is not an automorphism of a fibre but rather it is a twisted automorphism.
Given $\phi\in \Aut(G)$ let 
$$G(\phi) = \{(g,\phi)\st g\in G\}\cong G$$
be the corresponding ``twist'' of $G$.
It is a subset of the group $G\ltimes \Aut(G)$, which
acts on a standard fibre ($\cong G$) via $p\mapsto g\phi(p)$.
Given a path $\ga\in \Pi$ we get
$f(\ga)\in \Aut(G)$, and thus a twist of $G$:
$$G(\ga) := G(f(\ga))\subset G\ltimes \Aut(G).$$
We will sometimes (abusively) call such spaces 
$G(\phi)$
``twisted groups'', since the orbits of the conjugation action of $G\subset G\ltimes\Aut(G)$
on $G(\phi)$
are commonly called twisted conjugacy classes 
\cite{springer-twisted}.
The moduli space of framed $\cG$-local system is isomorphic to the  space of twisted representations
$$\THom(\Pi,G)\subset \Hom(\Pi,G\ltimes \Aut(G))$$
which is the subset where any 
path  $\ga$ is mapped to $G(\ga)$.
It is a smooth affine variety 
with an action of $G^\be$ changing the framings. 
Basic results \cite{AMW-DH, ABM-purespinors} then 
imply that 
$\THom(\Pi,G)$ 
is a {\em twisted} quasi-Hamiltonian $G^\be$-space---this notion is very close to that of a quasi-Hamiltonian space, the only difference in the axioms is that the moment map 
$\rho \mapsto \{\rho(\partial_i)\}$ now takes values in the twist $G^\be(\partial)$ of $G^\be$.
In turn this implies the resulting twisted character variety 
$$\MB(\wh\Si,\cG)=\THom(\Pi,G)/G^\be\cong \THom(\pi_1(\wh\Si),G)/G$$ 
is a Poisson variety with symplectic leaves obtained by fixing the twisted conjugacy classes of the local monodromies. 
The corresponding set-theoretic quotient is usually denoted 
$\HH^1(\wh \Si,\cG)$ or $\HH^1(\pi_1(\wh\Si),G)$, where $\pi_1$ acts on $G$ via $f$
(see \cite{frenkel-nabc, serre.galoiscoh}).

In this article we will put all this together (and extend to twisted formal types), 
by defining the notion
of twisted Stokes $\cG$-local systems and 
showing that: 

\begin{thm}\label{thm: thm1}
The moduli space $\THom_\IS(\Pi,G)$ of framed twisted Stokes $\cG$-local systems is a smooth affine variety and is a 
twisted quasi-Hamiltonian $\bH$-space.
\end{thm}
In turn this implies the resulting twisted wild character variety 
$$\THom_\IS(\Pi,G)/\bH$$ 
is a Poisson variety with symplectic leaves obtained by fixing the twisted conjugacy classes of the formal monodromies. 

The basic picture is as in \cite{gbs, p12} except now we might have a twisted local system in the interior of the curve, as well as a reduction to a twisted local system on each halo (and we may have twists on each halo even if there is no twist in the interior, as for Airy's equation).
Note that on any proper sector (at any pole) the local picture is isomorphic to that in \cite{gbs} so the main 
input here is to see that when these are glued together the result is indeed 
a twisted quasi-Hamiltonian space.

As in \cite{AMM} we may use fusion to do induction with respect to the genus and number of poles. Then as in \cite{gbs} we may use fusion within the halos to do induction with respect to the depth of the connections (reducing the order of poles)---for this it is essential to consider both types of twist simultaneously.
Then Theorem \ref{thm: thm1} follows   
from the special case of a connection with just one pole and one level on a disc, which is amenable to a direct proof. 
In this case (in fact for any connection with just one pole on a disc) 
the formula for the twisted quasi-Hamiltonian 
two-form is the same as that in \cite{gbs} 
(which in turn generalises \cite{saqh, fission}).

The intrinsic description of the Stokes data that we give rests heavily on the work of
many people, most directly Deligne, Loday-Richaud, Malgrange, Martinet--Ramis 
(and in turn Babbitt--Varadarajan, Balser--Jurkat--Lutz, Ecalle, Sibuya ...).

The layout of this article is as follows. 
The first few sections are preparatory, leading up the definition of
 twisted Stokes $\cG$-local systems in \S\ref{sn: sls}.
Then \S\ref{sn: classn} gives the classification of framed 
Stokes local systems.
Next \S\ref{sn: tqH} discusses twisted quasi-Hamiltonian geometry and proves Thm 
\ref{thm: thm1}.
Some examples are discussed at the end.

\noindent{\bf Acknowledgments.}
The first named author is grateful to Jochen Heinloth for useful conversations, and 
is partly supported by grants
ANR-13-BS01-0001-01 (Vari\'et\'es de caract\`eres et g\'en\'eralisations) 
and ANR-13-IS01-0001-01/02 
(Sym\'etrie miroir et singularit\'es irr\'eguli\`eres provenant de la physique).
The second named author is supported by JSPS KAKENHI Grant Number 15K17552.

\section{Torsors etc}\label{ssn: torsors}
We will briefly recall some elementary facts and notation for torsors, 
and $\cG$-local systems.
Let $G$ be a group.
A $G$-torsor is a set $\cP$ with a free transitive (right) action of $G$.
For example if $V$ is an $n$-dimension complex vector space, then the space 
$\cP=\Iso(\IC^n,V)$ of frames of $V$ is a $\GL_n(\IC)$-torsor---a point of $\cP$ is an isomorphism $\Phi:\IC^n\to V$, giving a basis of $V$, and $g\in\GL_n(\IC)$
acts on the right $\Phi\mapsto \Phi\circ g$, changing the basis of $\IC^n$.

A morphism $\phi:\cP_1\to\cP_2$ between two $G$-torsors is a $G$ 
equivariant map, 
i.e. such that $\phi(pg)=\phi(p)g$ for all $p\in \cP_1, g\in G$.
Any such map is an isomorphism.
Let $\Iso(\cP_1,\cP_2)$ denote the set of such maps, and let
$\Aut(\cP)=\Iso(\cP,\cP)$. 
Thus $\Aut(\cP)$ is a group that acts (freely and transitively) 
on $\cP$ on the left, commuting with the right $G$ action.
In this situation we say that $\cP$ is a $\Aut(\cP)$-$G$-bitorsor.
Similarly $\Iso(\cP_1,\cP_2)$ is an $\Aut(\cP_2)$-$\Aut(\cP_1)$-bitorsor.
If  $\cP$ has commuting, free, transitive actions of $G$ on the left and the right
then we will say $\cP$ is a $G$-bitorsor.

Let $\cF$ denote the standard/trivial $G$-torsor: it is just a copy of $G$ with $G$ acting by right multiplication.
Then $\Aut(\cF)=G$, acting by left multiplication on $\cF$.
The semi-direct product $G\ltimes \Aut(G)$ acts on $\cF$ via
$$(g,\phi)(p) = g\phi(p).$$
This extends the action of the normal subgroup $G\subset G\ltimes \Aut(G)$ acting by automorphisms of $\cF$.
(Here $\Aut(G)$ is the group of group automorphisms of $G$.)
We can view $G\ltimes \Aut(G)$ as a subgroup of the group $\Perm(\cF)$ of all bijections
$\cF\to \cF$ in this way.
If $\phi\in \Aut(G)$ let $\Ga\subset \Aut(G)$ be the subgroup generated by 
$\phi$ and write $\wt G=\wt G(\phi)=G\ltimes \Ga\subset G\ltimes \Aut(G)$.
Let
$$G(\phi) = \{(g,\phi)\st g\in G\} \subset \wt G(\phi)\subset G\ltimes \Aut(G)$$
be the component of $\wt G$ lying over $\phi$.
Then the natural left and right actions of 
$G\subset G\ltimes \Aut(G)$ are free and transitive, so
$G(\phi)$ is $G$-bitorsor.
Explicitly $(g_1,g_2)\in G\times G$ sends $(g,\phi)$ to  $(g_1g\phi(g_2),\phi)$.
In particular the conjugation action of $G$ is $\phi$-conjugation,
whose orbits are twisted conjugacy classes 
(cf. e.g \cite{springer-twisted}).
We will say that a $G$-bitorsor, such as $G(\phi)$, is a ``twist of $G$'', or, abusively, a ``twisted group''.

A {\em framing} of a $G$-torsor $\cP$ is an isomorphism
$\th\in \Iso(\cF,\cP)$. %
The choice of $\th$ is equivalent to choosing a point $p \in \cP$ (given 
$p\in \cP$ there is a unique $\th=\th_p$ such that $p= \th(1)$).
In turn there is an induced group isomorphism 
$$\eta=\eta_p:\Aut(\cP) \to G=\Aut(\cF),$$
namely
$\eta(m)=\th^{-1} \circ m\circ\th:\cF\to \cF$, for all $m\in \Aut(\cP)$.
Note that 
\beq
m(p) = p\cdot \eta(m)
\eeq
for all 
$m\in \Aut(\cP)$ and this characterises $\eta$.

\subsection{$\cG$-local systems} 

Now suppose $\Si$ is a connected space (typically a circle or a Riemann surface here).
A local system $\IL$ over $\Si$ is a covering map $\pi :\IL\to \Si$. The fibres are discrete, but may be uncountable.
If $G$ is a group then a $G$-local system $\IL$ is a local system 
such that each fibre is a $G$-torsor.
Said differently it is a sheaf which is a torsor under the constant sheaf $G$
over $\Si$
(so $\IL(U)$ is $G$-torsor for small open $U\subset \Si$). 
One can specify a $G$-local system 
via constant clutching  maps $g\in G$
$$\cF \cong \cF; \qquad p\mapsto g \cdot p$$
on each two-fold intersection of an open cover.
More generally one can replace $G$ by a local system $\cG$ of groups.
On a cover, such  $\cG$ is specified  by constant $\Aut(G)$-valued clutching maps $\phi$
$$ G \cong G; \qquad h\mapsto \phi(h)$$
on each two-fold intersection.
Then a $\cG$-local system $\IL$ is a sheaf which is a torsor under $\cG$
(so $\IL(U)$ is $\cG(U)$-torsor for small open $U\subset \Si$). 
On a cover, such  $\IL$ can be specified by constant $G\ltimes\Aut(G)$-valued clutching maps 
$(g,\phi)$ acting as
$$ \cF \cong \cF; \qquad p\mapsto g\cdot \phi(p)$$
on each two-fold intersection, where the $\phi$ are the automorphisms determining $\cG$.
If so, the local system of groups $\Aut(\IL)$ is specified by the clutching maps
$$ G \cong G; \qquad h\mapsto g\phi(h)g^{-1}.$$

\newq{}

\section{Graded local systems}\label{sn: grloc}

This section discusses local systems that are graded by a local system of sets, rather than by a set.

\subsection{Covers}

Let $\partial$ 
denote a circle, and let $\pi:I\to \partial$ 
be a covering map, i.e. a local system of sets.
Thus for each point $p\in \partial$ we have a discrete set 
$I_p=\pi^{-1}(p)$, and as $p$ moves around the circle the sets are locally constant
so that after one turn we obtain a monodromy automorphism $\si:I_p\to I_p$.
Upto isomorphism the monodromy $\si$ determines $I$.
In the cases of interest to us $I$ is a disjoint union of circles.

\subsection{Graded local systems}

Now let $V\to \partial$ be a local system of finite dimensional complex vector spaces. 
We will usually think of local systems of vector spaces as locally constant sheaves 
(but the equivalent viewpoint of a $C^\infty$ vector bundle with a flat connection
is sometimes helpful).

\begin{defn}
An ``$I$-graded local system'' is a local system $V\to \partial$
together with a point-wise grading
\beq\label{eq: grading}
 V_p = \bigoplus_{i\in I_p} V_{p,i} 
\eeq
of each fibre of $V$, by the set $I_p$, such that
each subspace $V_{p,i}$ is locally constant over
 $\partial$, via the local system structure of $V$.
\end{defn}

Thus if $p\in \partial$ is a basepoint then we have monodromy automorphisms
$\si\in \Aut(I_p),$ and $\wh\si\in \GL(V_p)$. 
The grading \eqref{eq: grading} is locally constant, but globally one has
$$\wh\si(V_{p,i}) = V_{p,\si(i)}$$
i.e. the monodromy of $V$ permutes the subspaces according to the monodromy of $I$.
In particular, in general, 
each subspace $V_{p,i}$ is not the fibre of a sub-local system of $V$. 

If $V$ is an $I$-graded local system, then we can associate an integer, the multiplicity, to each component $j\subset I$ of $I$, namely the rank of the corresponding subspaces of $V$:
\beq\label{eq: multj}
\mult(j) = \dim(V_{p,i}) \ge 0
\eeq
for any $i\in j_p\subset I_p$ and 
$p\in \partial$. 

\newq{}
\subsection{Local system of exponents} \label{ssn: exp ls}

Suppose $\Si$ is a complex curve and $0\in \Si$ is a smooth point.
Let $\pi:\wh \Si\to \Si$ denote the real oriented blow up of $\Si$ at $0$, and let 
$\partial=\pi^{-1}(0)\subset \wh \Si$ denote the circle of real directions emanating from $0$ in $\Si$. Open intervals $U\subset \partial$  parameterise 
(germs of) open sectors $\Sect(U)$ at $0$ with opening $U$.

Choose a local coordinate $z$ vanishing at $0$.
Let $\cI$ be the local system on $\partial$ whose local sections are 
complex polynomials in some root  $z^{-1/r}$ of $z^{-1}$ 
having constant term zero.
In other words local sections are
functions (defined on sectors) of the form 
\beq\label{eq: locsn q}
q = \sum_{i=1}^k  a_i z^{-i/r}
\eeq
where $a_i\in \IC, k,r\in \IN$.
(They are just the exponents of the exponential factors occurring in formal solutions of linear differential equations.)
Each local section $q$ has a well defined
ramification degree/index: it is the minimal integer $r=\ram(q)\ge 1$ 
such that $q\in \IC[z^{-1/r}]$.
Thus each local section $q$ of $\cI$ becomes single valued on the $r$-fold cyclic covering circle of $\partial$, where $r=\ram(q)$.
The {\em degree} of $q$ is the smallest integer $k=\deg(q)\ge 0$ 
possible in \eqref{eq: locsn q}.
The {\em level} of $q$ is 
$\lev(q)=\deg(q)/\!\ram(q)\in \IQ$.
Taking the exterior derivative of sections of $\cI$ yields an isomorphic local system 
``\,$d\,\cI$\,'', the {\em Deligne local system} \cite{deligne78},
whose local sections are one-forms. 

We will often use the equivalence between local systems and covering spaces, and thus view $\cI$ as a (vast) disjoint union of circles covering $\partial$.
Thus each component $I\in \pi_0(\cI)$  is a circle covering $\partial$, and we will write 
$\pi:I\to \partial$ for this finite covering (sub-local system).
Let $r=\ram(I)$ denote the degree of $\pi$.
If we choose an interval $U\subset \partial$  then $\pi^{-1}(U)$ has $r$ components, and each component determines a well defined function on $\Sect(U)$. 
These functions are the $r$ branches of a single function upstairs 
on a small annulus with boundary circle $I\to \partial$
(i.e. they form a single Galois orbit).
In other words if $b\in U$ is a basepoint and we choose a point 
$i\in I_b=\pi^{-1}(b)$ then we get a well defined function $q_i$ on $\Sect(U)$, and the analytic continuation of $q_i$ yields the other functions $q_j$ for  $j\in I_b$.

Thus choosing a component of $\cI$ is equivalent to choosing a Galois orbit of such functions. Thus for example  the Galois orbit $\{z^{-1/2}, -z^{-1/2}\}$ determines a component $I\subset \cI$, and we will sometimes write  $I=\langle z^{-1/2} \rangle$, bearing in mind that we also view $I$ as the two-fold covering circle of $\partial$ where these functions become single-valued.
The ramified circle drawn by Stokes can be identified with 
the circle $\langle 2x^{3/2}\rangle$, with $x=1/z$.

\begin{rmk}\label{rmk: Idefn}
Here is a coordinate independent definition of $\cI$.
For each interval $U\subset \partial$, let $\cF(U)$ be the algebra of (germs at $0$ of) 
holomorphic functions on $\Sect(U)$
which have finite monodromy (i.e. each $f\in \cF(U)$ has an analytic continuation around 
$0$ and returns to itself after a finite number of turns).
Let $\cF_m(U)\subset \cF(U)$ 
be the subalgebra of such functions which have at most a pole at zero 
(i.e. after analytic continuation they become meromorphic functions on a covering disc).
Let $\cF_h(U)\subset \cF_m(U)$ be the subalgebra of functions which, after analytic continuation, become holomorphic functions on a covering disc.
Then $\cI$ is %
defined by 
$\cI(U) = \cF_m(U)/\cF_h(U)$.
\end{rmk}

Deligne's way \cite{deligne78} of stating  the formal classification of connections on vector bundles %
is then as follows (see also  \cite{Mal-irregENSlong} Thm 4.4, \cite{BV89} \S4.8, and compare \cite{BJLformal}):

\begin{thm}\label{thm: deligne formal}
The category of connections on vector bundles on the formal punctured disc is equivalent to the category of $\cI$-graded local systems.
\end{thm}

Note that each fibre $\cI_p$ of $\cI$ is a lattice (a free $\IZ$-module), by adding functions.
Thus we can view $\cI$ as a local system of lattices over $\partial$.
In turn we can consider the pro-torus $\cT_p=\Hom(\cI_p,\IC^*)$ with character lattice $\cI_p$.
This pro-torus $\cT_p$ is the 
exponential torus of \cite{MR91} p.380.
Let  $\cT$ denote 
the resulting local system of pro-tori 
over $\partial$, so $\cT_p$ is the fibre of $\cT$ at $p\in \partial$.

Thus an $\cI$-graded local system $V$ is the same thing as a local system $V$ together with a morphism 
$\cT\to \Aut V$ of local systems of groups.
(Recall that a torus action on a vector space is the same as a grading by the character lattice of the torus. Here we do the same for locally constant vector spaces and locally constant pro-tori.
By ``morphism to $\Aut(V)$'' we mean one that factors through one of the natural algebraic quotients of $\cT$, corresponding to a finite rank local system of lattices $\subset \cI$.)

More generally we will use the following definition. 

\begin{defn}
An $\cI$-graded $G$-local system is
a $G$-local system $\IL$ together with a morphism
$$\cT\to \Aut(\IL)$$
of local systems of groups over $\partial$, factoring through an algebraic quotient of $\cT$.
\end{defn}

Here 
a $G$-local system is a right torsor 
under the constant sheaf of groups $G$, whereas a principal $G$-bundle is a right torsor 
under the sheaf of groups $U\mapsto G(\cO(U))$.
The generalisation of Theorem 
\ref{thm: deligne formal} also  holds 
(cf. \cite{BV83, MR91}):
\begin{thm}
The category of connections on principal 
$G$-bundles on the formal punctured disc is equivalent to the category of $\cI$-graded $G$-local systems.
\end{thm}

Similarly if we replace the constant group $G$ by 
a local system $\cG$ of groups over $\partial$
then we can define the notion of $\cI$-graded $\cG$-local system in the same way,
i.e. a $\cG$-local system $\IL$ plus a morphism $\cT\to \Aut(\IL)$. 
Similarly to above, 
one can show that they classify connections on torsors under non-split reductive groups over 
$\IC\flp z\frp$ (any such torsor is trivial and any such group is quasi-split and so determined by an outer automorphism of $G$).
When we speak of ``graded local systems'' we will always mean $\cI$-graded, unless otherwise stated.
 
 Note that the image of any morphism 
 $\cT\to \Aut(\IL)$ will be a twisted torus, 
 i.e. a local system $\IT\subset \Aut(\IL)$ of finite dimensional complex tori over $\partial$.
 Given a basepoint $b\in\partial$ the fibre $\IT_b$ of $\IT$ at $b$ occurs in the differential Galois group of the corresponding connection and is known as the {Ramis torus}.
 The character lattice of $\IT$ is a finite rank local system of free $\IZ$-modules $I\subset \cI$.
 The choice of such a quotient $\IT$ of $\cT$ is 
 determined by the choice of such $I\subset \cI$.
 A $\IZ$-basis of $I_b$ is just a certain  list $q_1,\ldots q_n$ of functions of the form \eqref{eq: locsn q} defined near $b$.

We will say that $\IL$ has ``twisted irregular type'' if $\IT$ is non-constant, i.e. 
if $I$ has non-trivial monodromy.

\subsection{Points of maximal decay}\label{ssn: pom}
Each circle $I=\langle q \rangle \subset \cI$ (except $\langle 0 \rangle$) has some special marked points on it:
a point $p\in I$ is a ``point of maximal decay'' if the corresponding function $q$ is such that 
$\exp(q)$ has maximal decay at $p$ (as $z\to 0$ in that direction).
We will call such points of maximal decay {\em p.o.m.}s, or {\em apples}.
There are $\deg(q)$ p.o.m.s on $I$.
For example if $q=\la z^{-i/n}$ they are the $i$ points where $q$ is real and negative.
In Stokes' Airy example the apples are the points $a,b,c$. They were singled out by him as the points where discontinuities in the constants occur (\cite{stokes1857} p.118), and this 
was generalised in the modern multisummation approach to the Stokes phenomenon that we use here (see e.g. \cite{MR91, L-R94, ramis-history}).

\newq{}
\subsection{Irregular types and irregular classes}

\begin{defn}
Two graded $\cG$-local systems
have the same ``bare irregular type'' or ``irregular class'' if they are locally isomorphic as graded local systems over $\partial$.
\end{defn}

If $\IL$ is a graded $\cG$-local system,
let $[\IL]$ denote the underlying irregular class/bare irregular type (i.e. its class under the above equivalence relation).

In more detail suppose $\IL^1,\IL^2$ are graded $\cG$-local systems so we have maps
$\phi^i:\cT\to \Aut(\IL^i)$ with images $\IT^i$ ($i=1,2$) say.
Then $\IL^1,\IL^2$ are locally isomorphic if they are isomorphic 
at any point $b\in\partial$, i.e. if there is an isomorphism
$\IL^1_b \to \IL^2_b$ (of $\cG_b$-torsors) such that the induced map 
$\Aut(\IL^1_b) \cong \Aut(\IL^2_b)$ makes the diagram 
\begin{equation}	\label{cd: lociso}
\begin{array}{ccc}
  \,\,\cT_b & = & \,\,\cT_b \\
\mapdown{\phi^1} && \mapdown{\phi^2} \\
  \Aut(\IL^1_b) & \mapright{\cong} & \Aut(\IL^2_b).  
\end{array}
\end{equation}
commute.
This implies that $\IT_b^1\cong \IT_b^2$ (as quotients of $\cT_b$), so correspond to the same (finite rank) lattice $I_b\subset \cI_b$.
Thus if $\cG$ is constant, an irregular class is determined by the quotient $\IT_b$ of  $\cT_b$ (for some $b$) plus the conjugacy class of the embedding 
$\IT_b\hookrightarrow G\cong \Aut(\IL_b)$.
Further if $G$ general linear group this amounts to remembering just the multiplicity of each component of $\cI$, as in \eqref{eq: multj}. In other words:

\begin{prop}\label{prop: gln irclass}
If $\cG=\GL_n(\IC)\times \partial$ then an irregular class is a 
map $\pi_0(\cI)\to \IN$ assigning a multiplicity $n_I\in \IN$ to each component $I\subset \cI$, such that $\sum n_I\cdot \ram(I)=n$.
\end{prop}

\begin{exercise}
First suppose $\cG$ is constant.
Let $\lt\subset \g$ be a Cartan subalgebra of $\g=\Lie(G)$ and let
$\wh\lt=\lt\flp z \frp\subset \g\flp z \frp$ 
be the corresponding Cartan subalgebra of the loop algebra.

1) Show that any element $Q\in \wh\lt/\lt\flb z \frb$
determines an irregular class $\bar{Q}$ and that any untwisted irregular class arises in this way.

Now let $\wh\lt \subset \g\flp z \frp$ be any Cartan subalgebra of the loop algebra.
Recall that one can take a root $w$ of $z$ so that 
$\wh\lt$ may be conjugated into $\lt\flp w \frp$ via $G\flp w \frp$ (note that Cartan subalgebras are conjugate over algebraically closed fields).

2) Show that the $w$-adic filtration of $\lt\flp w \frp$ 
yields a {\em canonical} filtration of $\wh\lt$.
Let $\wh\lt_{\ge 0}$ be the non-negative part of $\wh\lt$ for this filtration.

3) Show that any element $Q\in \wh\lt/\wh\lt_{\ge 0}$
determines an irregular class $\bar{Q}$ and that any irregular class arises in this way for some such $\wh\lt$ and $Q$.
\end{exercise}

\ 

Thus given a choice of Cartan subalgebra $\wh\lt$ we can define an 
``irregular type for $\wh\lt$'' to be an element 
$Q\in \wh\lt/\wh\lt_{\ge 0}$, and this determines an irregular class $\bar Q$.
This definition of twisted irregular type 
was suggested in \cite{gbs} Rmk 8.6.
More generally (if $\cG$ has finite monodromy) we can proceed similarly with Cartan subalgebras in the corresponding twisted loop algebra (since everything straightens out on passing to a cover).

\newq{}
\subsection{Framings and standard fibres}

Suppose $Q$ is an irregular class and $\cP\to \partial$ is a graded $\cG$-local system 
in the class $Q$.

Choose a basepoint $b\in \partial$ and fix an isomorphism $G\cong \cG_b$.
Then $\cP_b$ is a graded $G$-torsor: there is a map
$\varphi:\cT_b\to \Aut(\cP_b)$
of groups.
Let $\cF_0$ denote the trivial $G$-torsor, 
so $\Aut(\cF_0)=G$ acting by left translation.
Given a point $p\in \cP_b$ there is an isomorphism $\cF_0\to \cP_b$ of $G$-torsors 
and in turn an isomorphism $\Aut(\cP_b)\cong \Aut(\cF_0)=G$ of groups, 
and so  $\varphi$ becomes a map $\cT_b\to G$. 
Thus $\cF_0$ inherits a grading in this way.

\begin{defn}
A ``standard fibre $\cF$ for $Q$ at $b$'' is a grading of $\cF_0$
, i.e. a group map $\varphi:\cT_b\to \Aut(\cF_0)=G$,
such that $\varphi$ arises as above for some choice of $\cP$ and $p$.
\end{defn}

Now suppose we fix a standard fibre $\cF$ for $Q$ at $b$.
Then for any 
 graded $\cG$-local system $\IL$ in the class $Q$ we can make the following definition.

\begin{defn}
A framing of $\IL$ at $b$ is 
an isomorphism $\cF\to\IL_b$ of graded $G$-torsors,
i.e. a $\cT_b$-equivariant map of $G$-torsors,
where $\cT_b$ acts via the left actions of $\Aut(\IL_b)$ and $\Aut(\cF_0)$ respectively. 
\end{defn}

Let $\IT_b\subset G$ denote the image of $\cT_b$ and let 
$H\subset G$ be the centraliser of $\IT_b$.
Then $H$ is a connected complex reductive group, 
and the set of framings of $\IL$ at $b$ is an $H$-torsor.

Note that we could fix a maximal torus $T\subset G$ and 
define a ``very standard fibre'' to be a 
standard fibre with $\varphi$ landing in $T$.
This is often convenient since a 
map $\cT_b\to T$ corresponds to  a $\IZ$-linear map $X^*(T)\to \cI_b$.
In the untwisted case choosing a very standard fibre is equivalent to choosing an irregular type in the sense of \cite{gbs}.

\newq{}
\subsection{Gradings versus Reductions}\label{ssn: gvr}

Suppose $\cP\to \partial$ is an auxiliary graded $\cG$-local system, 
with irregular class $Q=[\cP]$.
Let $\IT\subset \Aut(\cP)$ be the image of $\cT$
and
let $\cH\subset \Aut(\cP)$ be the centraliser of 
$\IT$, i.e. $\cH=\GrAut(\cP)$ is the local system of graded automorphisms of $\cP$.
Now suppose $\IL$ is any $\cG$-local system over $\partial$.
Given the auxiliary graded local system $\cP$ we can make the following definition.

\begin{defn}
A ``flat reduction of $\IL$ to $\cH$'' is 
an $\cH$-local system $\IL_0$ 
 plus an isomorphism
$\chi: \IL\cong \IL_0\times_{\cH} \cP$
of $\cG$-local systems on  $\partial$.
\end{defn}

\begin{lem}\label{lem: gr vs redn}
Choosing a flat reduction 
of $\IL$ to $\cH$ is equivalent to 
choosing a grading of $\IL$ such that the resulting graded local system 
is in the same irregular class as $\cP$.
\end{lem}
\pf
Given a grading of $\IL$ 
in the class $Q$,
consider the local system $\IL_0:=\GrIso(\cP,\IL)$
of graded isomorphisms from $\cP$ to $\IL$.
Since $[\IL]=Q$, $\IL_0$ is nontrivial
and is an $\cH$-local system.
The natural map 
$\IL_0\times_\cH\cP\to \IL$ is an isomorphism, so a grading gives a reduction to 
$\cH$.
Conversely given a reduction $(\IL_0,\chi)$
then  $\IL_0\times_\cH \cP$ 
is graded (as $\IT$ lies in the centre of $\cH$, 
so acts on both sides of $\IL_0$), and isomorphic to $\IL$ via $\chi$.
\epf

Note that this is a generalisation/modification of the usual notion of reduction of structure group, since we twist by $\cP$ here (it reduces to the usual notion if $\cP$ is trivial as a 
$\cG$-local system).
For example in the untwisted case 
we can always take $\cP$ to be trivial as a $G$-local system,  
so gradings are then related to 
reduction of structure group in the usual sense, as in \cite{fission, gbs}. 
One could also avoid this by working with disconnected groups.

\newq{}
\subsection{Singular directions and Stokes groups}

Suppose $\IL\to \partial$ is a graded $\cG$-local system.
Thus we have a map $\cT\to \Aut(\IL)$ and 
so the Lie algebra bundle 
$\ad (\IL)=\Lie(\Aut(\IL))$ 
is an $\cI$-graded local system of vector spaces (from the $\cT$ action via the adjoint action of $\Aut(\IL)$).
Therefore for any $d\in \partial$ we may write 
\beq\label{eq: local roots}
\ad(\IL)_d = \bigoplus_{i\in \cI_d}\ad(\IL)_{d,i}.
\eeq
Recalling \S\ref{ssn: pom},    this has a distinguished subspace:
\beq\label{eq: stokes roots}
\isto_d:= 
\langle \ \ad (\IL)_{d,i} \st  
\text{ $i\in \cI_d$ is a point of maximal decay }
\rangle\subset \ad (\IL)_d.
\eeq
This is a nilpotent Lie subalgebra (cf. 
\cite{gbs} Lemma 7.3).
By definition the Stokes group in the direction $d$
is the corresponding connected unipotent group:
\beq
\ISto_d := \exp(\isto_d) \subset \Aut(\IL)_d.
\eeq
In turn the singular directions $\IA\subset \partial$ are the directions $d\in \partial$ such that  $\ISto_d$ has non-zero dimension. (They are sometimes called  ``anti-Stokes directions''.)

Said differently the pro-torus $\cT_d$ maps onto a torus $\IT_d$, which lies in some maximal torus $T\subset \Aut(\IL_d)\cong G$.
Dually $X^*(T)$, and thus each root of $T$, maps to $\cI_d$.
Then $\sto_d$ is the sum of the root spaces for the roots which map to a p.o.m.

For example if $V\to \partial$ is a local system of vector spaces, graded by $I=\langle x^{1/3}\rangle$, then $\Lie(\Aut(V))=\End(V)$ is graded by the degree $7$ cover
$\langle 0\rangle \cup\langle(1-\omega)x^{1/3}\rangle \cup\langle(\omega-1)x^{1/3}\rangle,$
where $\omega=\exp(2\pi i/3)$, containing $2$ apples.  
Note that in Stokes' example $\End(V)$ is graded by
the degree $3$ cover $\langle 0\rangle\cup\langle 4x^{3/2}\rangle$, 
and the Stokes diagram of 
$\langle 4x^{3/2}\rangle$ is the same as that of $\langle 2x^{3/2}\rangle$, so the points $a,b,c$ can be taken to be those defining $\IA$.

The grading \eqref{eq: stokes roots} of $\sto_d$ can be further refined by the levels:
each graded piece $\ad(\IL)_{d,i}$ has a level $\lev(i)\in \IQ$
so for any $k\in \IQ$ there is a level $k$ subspace $\sto_d(k)\subset \sto_d$. These
are Lie subalgebras and exponentiate to subgroups $\ISto_d(k)\subset\ISto_d$.
In turn this yields a direct spanning decomposition
\beq\label{eq: level deomp}
\ISto_d(k_1)\times\cdots  \times \ISto_d(k_r)\cong \ISto_d
\eeq
where $k_1< k_2<\cdots <k_r$ are the levels of the non-zero pieces occurring in 
\eqref{eq: local roots}.
Passing to a cover all this is isomorphic to that studied in \cite{gbs}, which
extends the $G$-Stokes data in \cite{bafi}, and  parts of \cite{MR91, L-R94} (which already handle the local twisted case for constant general linear groups).

\newq{}

\section{Twisted Stokes local systems} \label{sn: sls}

Let $G$ be a connected  complex reductive  group.
Let $\Si$ be a compact Riemann surface (possibly with boundary).
Let $\al\subset \Si$ be a finite subset of points (not on the boundary).
Let $\cG\to\Si^\circ$ be a local system of groups with each fibre isomorphic to $G$,
where $\Si^\circ = \Si\setminus \al$.
Let $\pi:\wh \Si\to \Si$ be the real oriented blow up of $\Si$ at each point of $\al$.
Let $\partial$ denote the boundary of $\wh \Si$, which we suppose is non-empty.
For each boundary circle $\partial_i\subset \partial$ 
let $\cT_i\to \partial_i$ denote the corresponding local system of pro-tori (which we take to be trivial if $\partial_i$ was already a boundary component of $\Si$).

Extend $\cG$ to $\wh \Si$ in the obvious way, and let $\cG_i$ 
denote the restriction of $\cG$ to $\partial_i$.

Choose a graded $\cG_i$-local system $\cP_i\to\partial_i$
(it is graded by $\cT_i$).
Let $Q_i=[\cP_i]$ be the corresponding irregular class and let $\bQ=\{Q_i\}$.

Write ${\bf \Si}=(\Si,\al,\bQ)$ for the resulting twisted bare irregular curve/wild Riemann surface, with structure group $\cG$. (This is convenient notation; as in
\cite{smid, bafi, gbs} the moduli of such $\bSi$ behaves like the moduli of the underlying Riemann surface.)

Define an auxiliary surface $\wt \Si$ as follows.
Let $\IA_i\subset \partial_i$ be the singular directions of $Q_i$ and let 
$\IA = \bigcup \IA_i\subset \partial.$
Let $\IH_i\subset \wh \Si$ be a tubular neighbourhood of $\partial_i$, the $i$th ``halo''. 
Let $\partial'_i\neq\partial_i$ denote the 
interior boundary circle of $\IH_i$.
For each $d\in \IA_i$ choose
$e(d)\in \partial'_i$ so that 
the points $e(d)$ are in the same order as the points $d\in\IA_i \subset\partial_i$.
Then  puncture
$\wh \Si$ at each point $e(d)$ to define 
\beq \label{eq: assoc.surface}
\wt\Si := \wh\Si \setminus \{e(d)\st d\in \IA  \}.
\eeq
Thus  there are surfaces and  maps:
$\wt\Si\into \wh \Si\onto \Si.$
We will sometimes call the punctures $e(d)$ ``tangential punctures''.
Their exact position is not important (it is sometimes convenient to draw non-crossing cilia between 
each $d$ and $e(d)$, which can be pulled tight whenever required, as in \cite{gbs}) 
A slightly different approach to defining 
$\wt \Si$ is possible, cf. \cite{p12} Apx. B.

Let $\cG$ also denote the restriction of $\cG$ to $\wt \Si$ (it has no monodromy around any tangential puncture).
For any $d\in\IA$ let $\ga_d$ be a
simple loop in $\wt \Si$  based at $d$ going once around the extra puncture $e(d)$ (and no others).

\begin{defn}
A twisted Stokes $\cG$ local system on $\bSi$ is a 
$\cG$-local system
$\IL\to\wt \Si$ 
such that 1) the restriction of $\IL$ 
to $\IH_i$ is graded
(with irregular class $Q_i$), and 2)
the monodromy of $\IL$ around 
$\ga_d$ is in the Stokes group
$\ISto_d\subset \Aut(\IL)_d$
for all $d\in \IA$.
\end{defn}

Here we extend $\cT_i$ from $\partial_i$ to $\IH_i$ in the obvious way
and use this to grade $\IL\bigl\vert_{\IH_i}$.
An isomorphism of Stokes local systems 
is an isomorphism of $\cG$-local systems on 
$\wt \Si$ which
restricts to an isomorphism of graded 
$\cG_i$-local systems on each halo.

We will say that $\IL$ is a ``twisted'' Stokes $\cG$-local system if any of the irregular types are twisted, or if $\cG$ is twisted. 
Otherwise it is ``untwisted''.

Using the auxiliary graded local system 
$\cP_i$ on $\IH_i$ (rather than just the classes $Q_i$)
then we can set $\cH_i=\GrAut(\cP_i)\to \IH_i$
so that a Stokes local system may be defined
equivalently 
(cf. Lemma \ref{lem: gr vs redn}) 
as a $\cG$-local system
$\IL$ on $\wt \Si$ 
together with a flat reduction to 
$\cH_i$ on $\IH_i$ for all $i$,
with monodromy around $\ga_d$ in $\ISto_d$
for all $d\in \IA$.

\begin{rmk}
Note the same definition makes sense for any complex affine algebraic group 
$G$. For groups over other fields we could replace the appearance of $\IC^*$ by 
$\IG_m$ in the definition of $\cT$.
\end{rmk}

\subsection{Framed moduli spaces}
Choose a basepoint $b_i\in\partial_i$ in each component of $\partial$.
Let $\be\subset \partial$ be the set of these basepoints.
Fix a framing of $\cG$ at each $b\in \be$, i.e. a 
group isomorphism $\cG_b\cong G$,
so any $\cG_b$-torsor is now canonically a $G$-torsor.
For each $b=b_i\in \be$, let 
$\cF_i=\cF_b$ be a standard fibre at $b$, isomorphic to  the 
fibre of $\cP_i$ at $b$, as graded $G$-torsors.
 Then a framing at $\be$ of a  Stokes $\cG$-local system $\IL$ is an
 isomorphism
 $$\cF_b\to \IL_b$$
of graded $G$ torsors for each $b\in \be$.
Let $H_i=\GrAut(\cF_i)\subset G$ be the group of graded automorphisms of $\cF_i$, and let
$\bH=\Prod H_i$. It is a connected complex reductive group. 

Let $\wMB(\bSi,\cG)$  denote the set of isomorphism classes 
of framed twisted Stokes $\cG$-local systems on $\bSi$.
Our next aim is to describe this more explicitly, and the next section prepares for this. 

\newq{}
\subsection{Formal monodromy of graded $\cG$-local systems}\label{ssn: formal monod}

Returning briefly to the set-up of section \ref{sn: grloc},
let $\IL\to\partial$ be a graded $\cG$-local system.
Choose a basepoint $b\in \partial$.
Then there is a distinguished subset 
$$H(\partial)\subset \Perm(\IL_b)$$
of the set of all bijections $\IL_b\to \IL_b$, where the monodromy of $\IL$ lives.
We will describe it here and show it is a twist of the group $H=\GrAut(\IL_b)$.
First we give the abstract approach, then a more concrete one (involving more choices).

The fibre $\IL_b$ is a $\cG_b$-torsor.
Let $\IT_b\subset \Aut(\IL_b)$ be the Ramis torus (the image of $\cT_b$).
Then $H=\GrAut(\IL_b)$ is the centraliser of $\IT_b$.
Given $t\in \IT_b$ let $t'\in \IT_b$ be the element obtained by parallel translating $t$ around $\partial$ (once in a positive sense).
Similarly $g\in \cG_b$ yields $\phi(g)\in \cG_b$
and $p\in \IL_b$ yields $\mon_\partial(p)\in \IL_b$.
Thus $\phi\in\Aut(\cG_b)$ and $\mon_\partial\in \Perm(\IL_b)$.
Since everything is locally constant the following hold (where $m=\mon_\partial$):

1) $m(pg)=m(p)\phi(g)$ for all $p\in \IL_b, g\in \cG_b$, and 

2) $m\circ t = t' \circ m\in \Perm(\IL_b)$.

The first of these says that $m$ is in a twist of $\Aut(\IL_b)\cong G$ determined by $\phi$, and 2) says $m$ is in a certain subset of this twist.
Let $H(\partial)\subset  \Perm(\IL_b)$
be the set of elements $m$ satisfying 1) and 2).

\begin{lem} \label{lem: Htwist}
The set $H(\partial)$
of elements $m$ satisfying 1) and 2) 
is a twist of $H$. 
\end{lem}
\pf 
Clearly $H\times H$ acts, via $(h_1,h_2)(m)= h_1mh_2$.
Also if $m_1,m_2$ are two solutions, then 1) implies $m_1m_2^{-1}\in \Aut(\IL_b)$
and 2) then implies $m_1m_2^{-1}\in H$,
so the space of solutions is indeed an $H$-bitorsor.\epf

To make this more explicit
fix an isomorphism $G\cong \cG_b$,
a standard fibre $\cF$, and a graded isomorphism 
$\cF\cong \IL_b$.
Then we can identify $\phi\in \Aut(G)$,  
and the Ramis torus becomes a torus $\IT_b\subset G=\Aut(\cF)$,
which is equipped with its monodromy automorphism $t\mapsto t'$.
The set of
elements $m$ satisfying 1) becomes the twist 
$G(\phi)=\{(g,\phi)\st g\in G\}\subset G\ltimes \Aut(G)$,
so any $m$ satisfying 1) can be written as $(M,\phi)\in G(\phi)$.
In turn the solutions to 1) and 2) becomes the subset
$H(\partial)\subset G(\phi)$ 
of elements $(M,\phi)$  such that
$(M,\phi)(t,1) = (t',1)(M,\phi)$,
i.e. 
\beq\label{eq: concrete Mcondn}
M\in G \quad\text{ such that }\quad M\phi(t) = t' M\in G
\eeq
for all $t\in \IT_b$.
Further, if we have an auxiliary graded $\cG$ local system $\cP$ in the same class as $\IL$, with monodromy $(P,\phi)$ say, then \eqref{eq: concrete Mcondn} says that 
$t'=P\phi(t)P^{-1}$
so the constraint on $M$ is that $MP^{-1}$
commutes with every element of $\IT_b$, i.e. 
\beq \label{eq: very concrete Mcondn}
 M\in HP = P\phi(H) \subset G.
 \eeq 
Note that $H(\partial)\subset G(\phi)$ doesn't depend on the specific graded $\cG$-local system chosen, only the irregular class.

\begin{rmk}
Equivalently we can consider the group $\wt G:= G\ltimes \Ga\subset G\ltimes \Aut(G)$ 
where  $\Ga=\langle \phi\rangle \subset \Aut(G)$, so that $G(\phi)$ is a component
of $\wt G$, and $\IT_b\subset G$ is a torus in the identity component of $\wt G$.
Then $H(\partial)\subset G(\phi)$ is
a distinguished component of the normaliser of $\IT_b$ in $\wt G$.
\end{rmk}

In particular this enables us to give a more explicit 
approach to graded $\cG$-local systems.
Suppose $\IL$ is a $\cG$ local system with monodromy $(M,\phi)\in G(\phi)\subset \Perm(\IL_b)$.
Then to grade $\IL$ we need a morphism $\IT\into \Aut(\IL)$,
where $\IT$ is a local system of tori, which is given as an algebraic quotient of $\cT$.
Such $\IT$ is determined by the corresponding local system 
$I=X^*(\IT)\subset \cI$ of finite rank lattices.
Thus the first choice is such $I$ (a finite rank local system of lattices, given as a sublocal system of $\cI$). 
This is determined by the lattice $I_b\subset \cI_b$ in any fibre.
The only restriction on $I_b$ is that it is closed under monodromy.
The monodromy of $I$ determines the monodromy $t\mapsto t'$ of $\IT$.
Given $\IT$ we need to embed it in $\Aut(\IL)$.
This amounts to the compatibility condition $t'=M\phi(t)M^{-1}$ given in 
\eqref{eq: concrete Mcondn}.

If we now pass to a finite cyclic cover to kill the monodromy of $\IT$ 
(see Apx. \ref{apx: untwisting})
by setting $z=w^r$,
then the embedding $I_b=X^*(\IT_b)\into \cI_b$
lands in $\IC\flp w \frp/\IC\flb w \frb$.
This $\IZ$-linear map amounts to a $\IC$-linear map 
$\Lie(\IT_b)^*\to \IC\flp w \frp/\IC\flb w \frb$,
and thus to an element
$$Q\in \lt_1\flp w \frp/\lt_1\flb w \frb$$
where $\lt_1= \Lie(\IT_b)$.
It satisfies 
\beq\label{eq: liftedQcondn}
Q(\ze w) = M\phi(Q(w))M^{-1}
\eeq
where $\ze=\exp(2 \pi i/r)$.
If we choose a maximal torus $T\subset G$ with Lie algebra $\lt$ containing $\lt_1$, then such $Q$ is 
a special case of the irregular types considered in \cite{gbs} (on the $w$-disc).

Conversely suppose we have $Q\in \lt\flp w \frp/\lt\flb w \frb$
satisfying \eqref{eq: liftedQcondn} for some $\phi\in \Aut(G)$ and $M\in G$.
Then $Q$ defines a lattice map
$$\langle Q,- \rangle: X^*(T) \to \IC\flp w \frp/\IC\flb w \frb.$$
Thus there is a unique subtorus $\IT_b\subset T$ such that $X^*(\IT_b)$
is the image of the map $\langle Q,- \rangle$. 
By construction $\IT_b$ is a quotient of $\cT_b$.
Then \eqref{eq: liftedQcondn} allows us to descend $\IT_b$ to a local system of tori $\IT$, and we can define a $\cG$-local system $\IL$ with clutching map 
$(M,\phi)\in G(\phi)$, and \eqref{eq: liftedQcondn} implies $\IL$ 
is graded by $\IT$.
Here $\cG$ is determined by $\phi$.
So $Q$ determines an irregular class in this way, and all irregular classes arise in this fashion.

\begin{rmk}\label{rmk: Hisos}
Note that, given the auxiliary graded $\cG$-local system $\cP$, choosing $\IL$ corresponds to choosing an $\cH$-local system $\IL_0\to \partial$ as in \S\ref{ssn: gvr}, where 
$\cH=\GrAut(\cP)$.
As above
the monodromy of any such $\IL_0$ lives in a twist $H(\psi)$ of $H$, where 
$H=\cH_b$ and $\psi\in \Aut(H)$ is the monodromy of $\cH$.
For consistency we should check that $H(\psi)$ is isomorphic to the twist $H(\partial)$ 
appearing above (as bitorsors for $H$).
Indeed, in the notation above $\psi=P\phi(\cdot)P^{-1}$, and 
$H(\psi)$ consists of elements $(h,\psi)$, whereas $H(\partial)=\{(hP,\phi)\st h\in H\}$.
The action of $(h_1,h_2)\in H\times H$ then takes $h$ to $h_1h\psi(h_2)$ in both cases, so they are indeed isomorphic twists of $H$.
\end{rmk}

\newq{}

\section{Classification of Stokes local systems}\label{sn: classn}

Let $\Pi=\Pi_1(\wt\Si,\be)$
denote the fundamental groupoid of $\wt\Si$
with the given basepoints.
Let $f:\Pi\to \Aut(G)$ be the monodromy of $\cG$ 
(using the chosen identifications $\cG_b\cong G$ for all $b\in \be$).
For any $\ga\in \Pi$ let 
$G(\ga)\subset G\ltimes \Aut(G)$ denote the twist of 
$G$ lying over $f(\ga)$.

Given a framed Stokes local system 
and a path $\ga\in \Pi$ from $b_1$ to $b_2$ say,
then (using the framings) we get  
a bijection 
$\mon_\ga:\cF_1\cong \cF_2$
where $\cF_i$ is the standard fibre at $b_i$.
Since $\IL$ is a $\cG$-local system 
$\mon_\ga$ lives in the twist $G(\ga)$ of $G$.
Thus a framed Stokes local system determines a twisted representation
$$\rho\in \THom(\Pi,G):=\{ \rho\in
\Hom(\Pi,G\ltimes \Aut(G))\st \rho(\ga)\in G(\ga) \text{ for all }\ga\in \Pi\}.$$
Note that, by projecting onto the $G$ factor, $\THom(\Pi,G)$ could also be defined as 
the space of maps $\rho:\Pi\to G$ satisfying the twisted composition rule
$\rho(\ga_1\circ \ga_2)= \rho(\ga_1) \phi_1(\rho(\ga_2))$ for composable paths,
where $\phi_1=f(\ga_1)$.

We will view any boundary 
circle $\partial_i\subset \partial$ 
as a loop based at $b_i$,
oriented in the positive sense, and let $\bar\partial_i=\partial_i^{-1}$ denote the opposite loop (oriented in a negative/clockwise sense).
Thus each boundary component determines a twist 
$G(\partial_i)$ of $G$.
Further, as in \S\ref{ssn: formal monod},
there is a distinguished subset  
$H(\partial_i)\subset G(\partial_i)$,
which is a twist of $H_i$.

If  $d\in \IA_i$ let
$\la_d \subsetneq \partial_i$
be the arc 
from $b_i$ to $d$ in a positive sense.
By parallel translation along $\la_d$, we identify the Stokes group 
$\ISto_d\subset \Aut(\IL_{d})$ as a subgroup of
$\Aut(\IL_{b_i})$, and thus as a subgroup of $G=\Aut(\cF_i)$.
Also let
$$\wh\ga_d := 
\la^{-1}_d\circ \ga_d \circ  \la_d\in \Pi$$
be the simple loop around 
$e(d)$ based at $b_i$, 
where $\ga_d$ is the loop around $e(d)$ based at $d$ considered above.
 Note that $\rho$ has the following two properties:

1) $\rho(\partial_i)\in H(\partial_i)\subset 
G(\partial_i)$ for any boundary circle $\partial_i$,
as in \S \ref{ssn: formal monod}, and

2) $\rho(\wh\ga_d)\in \ISto_d\subset G$
for each singular direction $d\in \IA$.

\begin{defn}
A twisted representation 
$\rho\in \THom(\Pi,G)$ is a ``Stokes representation'' if conditions 1) and 2) hold.
\end{defn}

Let $\THom_\IS(\Pi,G)\subset \THom(\Pi,G)$ denote the subset of Stokes representations.

Note that the group $\bH\subset G^\be$ 
acts on $\THom_\IS(\Pi,G)$ 
as follows: $\{k_i\} \in \bH $
sends $\rho$ to $\rho'$, where
$\rho'(\ga)= k_j\rho(\ga)k_i^{-1}$
for any path $\ga\in \Pi$
from $b_i$ to $b_j$.
Recall that $G$ is embedded %
in  $G\ltimes\Aut(G)$ and $H_i\subset G$. 
The following is now straightforward.
\begin{prop}
Taking monodromy yields a bijection 
$\wMB(\bSi,\cG)\cong \THom_\IS(\Pi,G)$ from 
the isomorphism classes of framed twisted Stokes local systems  to the twisted Stokes representations.
Moreover the  action of $\bH$ corresponds to changing the framings and the set of isomorphism classes of twisted 
Stokes local systems  corresponds bijectively to 
the set of $\bH$ orbits in $\THom_\IS(\Pi,G)$. 
\end{prop}

\section{Twisted quasi-Hamiltonian geometry} \label{sn: tqH}
Our conventions for complex quasi-Hamiltonian geometry are as in \cite{gbs},
except here it is convenient to work with twisted quasi-Hamiltonian spaces.

Let $G$ be a connected complex reductive group with an automorphism $\phi\in \Aut(G)$.
Fix a symmetric non-degenerate complex bilinear form $(-,-)$ on $\g=\Lie(G)$, invariant
under the adjoint action and under $\phi$. 
Let $G(\phi)\subset \wt G=\wt G(\phi)\subset G\ltimes\Aut(G)$ be the twist of $G$ determined by $\phi$, on which $G$ acts by $\phi$-twisted conjugation.

\begin{defn}
A twisted quasi-Hamiltonian $G$-space (tq-Hamiltonian)
is a complex manifold $M$ with an action of $G$, an invariant 
holomorphic two-form $\omega$ and a $G$-equivariant map $\mu:M\to G(\phi)$, to a twist of $G$
(with the twisted conjugation action),
which satisfies the axioms of quasi-Hamiltonian $\wt G$-space for the action of 
$G\subset \wt G$. 
\end{defn}

Thus $M$ has a  twisted group valued moment map.
In brief such $M$ is almost a quasi-Hamiltonian $\wt G$-space (considered in 
\cite{ABM-purespinors}), 
but we only insist the identity component of $\wt G$ acts.
Alternatively it is a special type of quasi-Hamiltonian $\g$-space 
(also considered in \cite{ABM-purespinors})
where we insist the identity component of $\wt G$ does act.
Needless to say all the basic properties follow from those established in \cite{AMM, AKM, ABM-purespinors}.
In particular:

$\bullet$\ Any twisted conjugacy class $\cC\subset G(\phi)$ is a tq-Hamiltonian $G$ space.

$\bullet$\ If $\mu_i:M_i\to G(\phi_i)$ are tq-Hamiltonian moment maps for $i=1,2$, then the fusion
$M:=M_1\fus M_2$ is a tq-Hamiltonian $G$-space 
with moment map  $\mu_1\cdot\mu_2:M\to G(\phi_1\phi_2)$.

$\bullet$\ If $\phi_1$ is the inverse of $\phi_2$, then the gluing
$M_1\glue{} M_2 = M_1\fus{} M_2\spq G$ is well defined.

$\bullet$\ If $M$ is an affine variety and a tq-Hamiltonian $G$-space, then the affine quotient
$M/G$ is a Poisson variety, with symplectic leaves $M\spqa{\cC}G=\mu^{-1}(\cC)/G$, for twisted conjugacy classes $\cC\subset G(\phi)$.

As another example consider $H=G\times G$ with automorphism $\chi$ switching the two factors.
Then the group $G$ is isomorphic to the twisted conjugacy class in $H(\chi)$ containing 
$((1,1),\chi)$, and so is a tq-Hamiltonian $H$-space.
Then $G\fus G$ is a q-Hamiltonian $H$-space (since $\chi^2=1$), 
and is isomorphic to the double $\bD=\bD(G)$ of \cite{AMM} (which is the framed $G$-character variety $\wMB$ attached to an annulus). 
This derivation of the double from a conjugacy class is from \cite{AMW-DH}.
Similarly $G(\phi)$ is a tq-Hamiltonian $H$ space for any $\phi\in \Aut(G)$, where now
$\chi(g_1,g_2)=(\phi(g_2),\phi^{-1}(g_1))$.
Thus the ``twisted double'' $\bD(\phi,\psi):= G(\phi)\fus G(\psi^{-1})$
is also tq-Hamiltonian $H$-space. 
The moment map takes values in the twist $G(\phi\psi)\times G(\phi^{-1}\psi^{-1})$ of $H$,
and so we can fuse the two factors to obtain the ``twisted internally fused double''
$\ID(\phi,\psi)$ which is a tq-Hamiltonian $G$-space, whose moment map takes values in 
$G(\phi\psi\phi^{-1}\psi^{-1})$.
This is the framed twisted tame character variety $\wMB$ attached to a one holed torus, with arbitrary twist. %

Now return to the setup of \S\ref{sn: classn}.
Let $\bH(\bar\partial) = \Prod H(\bar\partial_i)\subset\Prod G(\bar\partial_i)$, 
which is a twist of $\bH$, and let 
$$\mu : \THom_\IS(\Pi,G) \to \bH(\bar\partial);
\qquad \rho \mapsto \{\rho(\bar\partial_i)\}$$
be the map taking
the (inverse of the) formal monodromies. 

Choose 
a locally constant, monodromy-invariant, 
non-degenerate invariant symmetric $\IC$-bilinear form on the Lie algebra of each fibre of 
$\cG\to \wt\Si$.
This determines an invariant bilinear form on 
$\Lie(\bH)$.

The main result then is as follows.

\begin{thm}\label{thm: main}
The space $\THom_\IS(\Pi,G)$ of twisted Stokes representations is a 
twisted quasi-Hamiltonian $\bH$-space with moment map $\mu$.
\end{thm}

Let $\MB(\bSi,\cG)$ denote the affine variety associated to the ring of $\bH$ invariant functions on $\THom_\IS(\Pi,G)$.

\begin{cor}	
The twisted wild character variety $\MB(\bSi,\cG)$ is an algebraic Poisson variety.
The symplectic leaves of $\MB(\bSi,\cG)$, the symplectic twisted wild character varieties, are the multiplicative symplectic quotients
$$\MB(\bSi,\cG,\cC) = \THom_\IS(\Pi,G)\spqa{\cC}\bH
 = \mu^{-1}(\cC)/\bH$$
 for twisted conjugacy classes 
 $\cC\subset \bH(\bar\partial)$.
\end{cor}

\newq{}
\subsection{Discs and pullbacks}\label{ssn: discs}
Now we will consider the simplest example of a disc with one marked point, and 
explain how to relate the Stokes data to the untwisted case.

Let $\Delta$ be a closed complex disc with a marked point $0$ in its interior.
Let $\cG\to \Delta\setminus 0$ be a local system of groups. 
Choose an irregular class $Q$ at $0$, and let 
$\wt \De\hookrightarrow \wh \De \onto \De$ be the associated surfaces, as in 
\eqref{eq: assoc.surface}.

Choose basepoints $b_0,b_1\in \partial\wh\De$  
(with $b_0$ lying over $0$), and standard fibres $\cF_0,\cF_1$.
To simplify notation let $\partial$ be the boundary 
circle of $\wt \De$ over $0$.
Choose framings of $\cG$ at $b_0,b_1$ and let $\phi\in\Aut(G)$ be the 
monodromy of $\cG$ around $\partial$.
Let $$\cA(Q)=\THom_\IS(\Pi,G)$$ 
denote the resulting space of Stokes representations.

Let $\IT_0\subset H\subset G=\Aut(\cF_0)$
denote the associated groups at $b_0$, 
let $\wt H\subset \wt G=\wt G(\phi)$ be the 
normaliser of $\IT_0$ in $\wt G$
and let 
$H(\partial)\subset G(\phi)$
denote the distinguished component of $\wt H$.
Let $\ISto_d\subset G=\Aut(\cF_0)$ denote the Stokes groups for each 
singular direction
$d\in \IA\subset \partial$ (translated to $b_0$ as above).
They are unipotent groups normalised by $H$.

\begin{figure}[ht]
	\resizebox{6cm}{!}{\includegraphics[width=5cm]{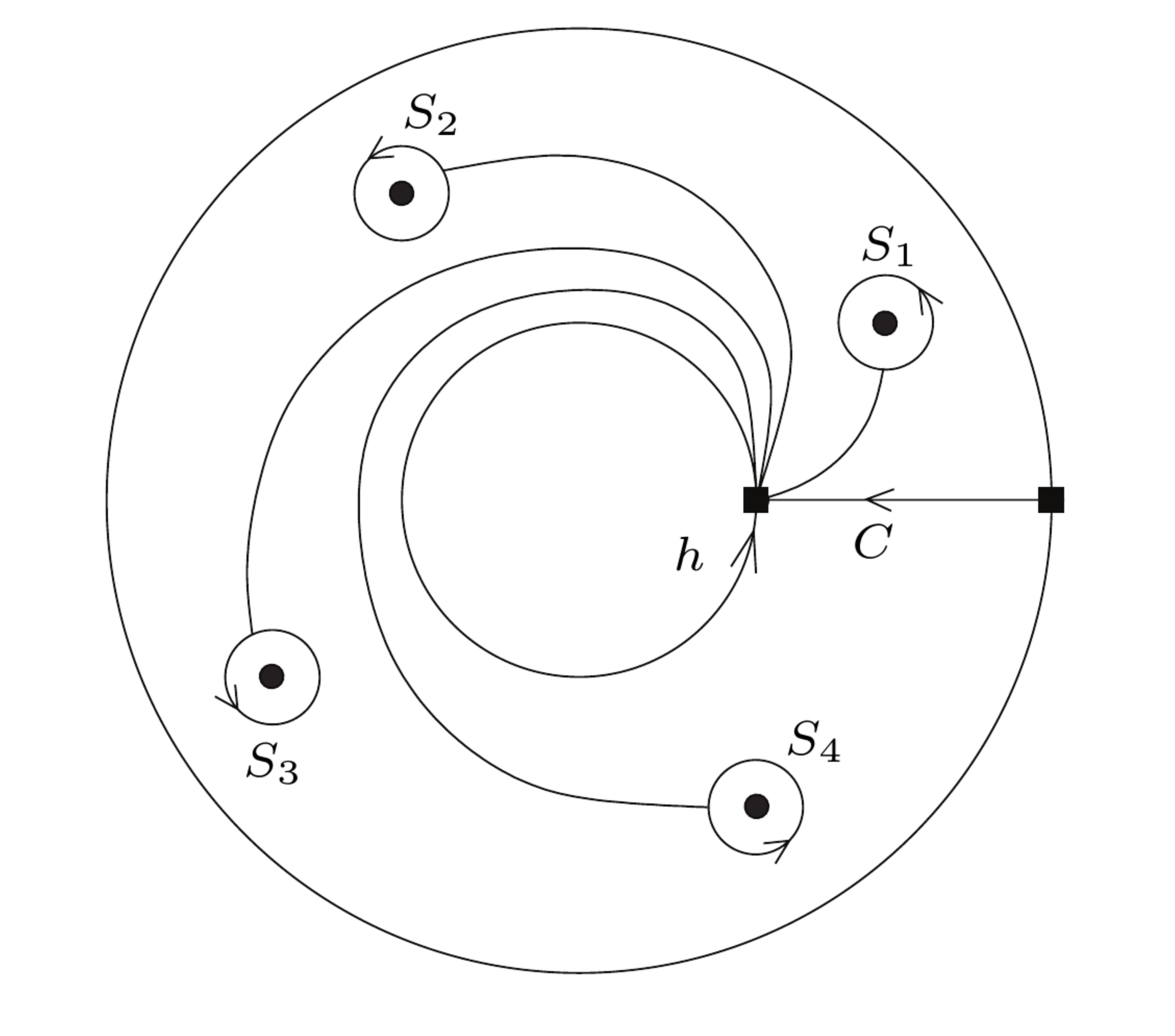}}
\end{figure}

Choosing paths generating $\Pi$ as in the figure above,
yields the following description:

\begin{lem}
There is an isomorphism
$\cA(Q) \cong G\times H(\partial) \times \prod_{d\in \IA} \ISto_d.$
In particular $\cA(Q)$ is a smooth affine variety.
\end{lem}

Here we suppose the framings of $\cG$ are chosen so that $\cG$ has no monodromy along the chosen path from $b_0$ to $b_1$.
A point of $\cA(Q)$ will be denoted $(C,h,\bS)$ with $C\in G, h\in H(\partial)$ and 
$\bS\in \prod_{d\in \IA} \ISto_d$ where $\bS = (S_1,\ldots,S_{s})$ and
we label the $s=\#\IA$ points of $\IA$ in a positive sense from $b_0$. 
The action of $G\times H$ on $\cA(Q)$ is given by 
$$(g,k)(C,\bS,h)  = (kCg^{-1}, k\bS k^{-1}, khk^{-1})$$
where $(g,k)\in G\times H$ and $k\bS k^{-1} = (kS_1k^{-1},\ldots,kS_{s}k^{-1})$.
Define an algebraic two-form $\omega$ on $\cA(Q)$ by 
\beq  \label{eq: omeg def}
2 \omega = 
(\bar\ga,\Ad_b\bar\ga) 
+ (\bar\ga,\bar\be)
+ (\bar\ga_s, \ch) -
\sum_{i=1}^s(\ga_{i},\ga_{i-1})
\eeq
where 
$\ga_i = C_i^*(\th),\bar\ga_i = C_i^*(\bar\th),
\eta = h^*(\th_H),\bar\be = b^*(\bar\th)$,
where $\th, \bar \th$ are the Maurer--Cartan forms on $\wt G$ (and 
$\th_H, \bar \th_H$ are the Maurer--Cartan forms on $\wt H$),
where 
$C_i: \cA(Q) \to G$ is the map defined by
$C_i = S_i\cdots S_2S_1C$ and
 $b = hS_{s}\cdots S_2S_1: \cA(Q) \to G(\phi)$.
Note this expression for $\omega$ is the same as in \cite{gbs} eqn. (9).

\begin{thm}\label{thm: qh disc}
$\cA(Q)$ is a twisted quasi-Hamiltonian $G\times H$-space, with two-form $\omega$ and twisted group valued moment map
$$\mu(C,\bs,h) = (C^{-1} h S_{s}\cdots S_2 S_1 C,\,\, h^{-1}) \in 
G(\partial)\times H(\bar\partial).$$
\end{thm}

We will prove this as in \cite{gbs} by induction with respect to the number of levels.
Recall the decomposition \eqref{eq: level deomp} of the Stokes groups
in terms of the levels $k_1<\cdots <k_r$.

\begin{prop} \label{prop: 1level}
Suppose Thm. \ref{thm: qh disc}
holds in all the cases with only one level $k_1$.
Then Thm. \ref{thm: qh disc} holds in general.
\end{prop}
\pf
This follows exactly as in \cite{gbs} Prop. 7.12.
The sequence of groups $H=H_1\subset \cdots H_r$ of \cite{gbs} (33)
is now the fibre at $b_0$ of a sequence $\cH=\cH_1\subset \cdots\subset  \cH_r$ 
of local systems of groups, which may be defined as follows.
For each $k\in\IQ$ let $\cI^{k}\subset \cI$ be the sublocal system (of lattices) of level $<k$.
Then the map $\cI\to \cI/\cI^{k}$ determines a sublocal system $\cT^k\subset \cT$ of pro-tori.
Thus if $\cP$ is graded by $\cT$, then it is graded by $\cT^k$ for any $k$.
Thus there is a sequence 
$$\IT^r\subset \cdots \subset\IT^2\subset  \IT^1=\IT$$
of local systems of tori,
with $\IT^i$ the image of $\cT^{k_i}$  in $\Aut(\cP)$.
Then $\cH_i$ is the centraliser in $\Aut(\cP)$ of $\IT^i$.
The story then proceeds as in \cite{gbs}, nesting/gluing the
 twisted quasi-Hamiltonian spaces together for each level.
\epf

\begin{prop}
Thm. \ref{thm: qh disc}
holds in all the cases with only one level $k_1$.
\end{prop}

\pf
Let $\psi=P\phi(\cdot)P^{-1}\in \Aut(G)$, as in Rmk \ref{rmk: Hisos}.
Let $\ISto_{is+j} := \psi^{-i}(\ISto_j)$ for $i \in \IZ,\ j=1,2, \dots ,s$.
Then since $\IT$ has finite monodromy (given by $\psi$) we can pull 
back to a finite cyclic cover (as in Apx. \ref{apx: untwisting})
and observe (as in  \cite{gbs} proof of Lemma 7.11)
there exists $l \in \IZ_{>0}$ such that for any $j \in \IZ$:
\begin{enumerate}
\item $\ISto_{j+2l}=\ISto_j$;
\item the subgroups $\ISto_{j+1}, \dots , \ISto_{j+l}$
directly span the unipotent radical $U_j$ of a parabolic subgroup of $G$
with Levi subgroup $H$, and $\ISto_{j+l+1}, \dots , \ISto_{j+2l}$ directly span 
the unipotent radical of the opposite parabolic. 
\end{enumerate}
Then the Lie algebras $\ls_j := \sto_j=\Lie (\ISto_j)$ and $\lh := \Lie(H)$ satisfy
\[
(\ls_i,\ls_j)=0 \quad (\lvert i-j \rvert <l), 
\qquad \g = \lh \oplus \bigoplus_{m=j-l+1}^{j+l} \ls_m \quad (j \in \IZ),
\]
and hence $\ls_j$ and $\ls_{j+l}$ are dual to each other with respect to the pairing
for each $j \in \IZ$. Thus we have
\[
(\lh \oplus \ls_j )^\perp = \bigoplus_{m=j-l+1}^{j+l-1} \ls_m \quad (j \in \IZ).
\] 
The proof of (QH1), (QH2) is similar to the untwisted case (see \cite{gbs}).
For (QH3),
fix $p \in \mathcal{A}$ and $v \in \Ker \omega \cap \Ker d\mu \subset T_p \mathcal{A}$.
Then  for arbitrary $u \in T_p \mathcal{A}$ we have
\[
0=2\omega(v,u) = 2(\Ad_h^{-1} \bar{\gamma}',\dot{\eta}) + \sum_{i=1}^s (\Delta_i,\dot{\sigma}_i)
\]
using the notation from \cite{gbs} (noting that $h$ is in the group $\wt G$) where
\begin{equation}\label{eq:delta}
\Delta_i = \Ad_{h[si]}^{-1} \bar{\gamma}' + \Ad_{[i-1,1]}\bar{\gamma}'
-\sum_{j>i} \Ad_{[ji]}^{-1} \bar{\sigma}'_j + \sum_{j<i} \Ad_{[i-1,j]} \sigma'_j.
\end{equation}
We see that $\pi_{\lh}(\bar{\gamma}')=0$ 
and $\Delta_i$ is orthogonal to $\ls_i$.
Note that $d\mu(v)=0$ implies 
\[
\Ad_b^{-1} \bar{\gamma}' - \bar{\gamma}' = \beta' = \sum_{j=1}^s \Ad_{[j1]}^{-1}\bar{\sigma}'_j.
\]
Hence
\[
\Ad_{h[si]}^{-1} \bar{\gamma}' - \Ad_{[i-1,1]}\bar{\gamma}' = \sum_{j>i} \Ad_{[ji]}^{-1}\bar{\sigma}'_j + 
\sum_{j \leq i} \Ad_{[i-1,j]} \sigma'_j.
\]
By \eqref{eq:delta} and the above equality 
we see that $R_i := \Delta_i + \sigma'_i$, $i=1,2, \dots, s$ 
have the following expressions:
\begin{align}
R_i 
= 2 \left( \Ad_{[i-1,1]}\bar{\gamma}' + \sum_{j \leq i} \Ad_{[i-1,j]} \sigma'_j \right) 
= 2\left( \Ad_{h[si]}^{-1} \bar{\gamma}' - \sum_{j>i} \Ad_{[ji]}^{-1} \sigma'_j \right).
\end{align}
In particular, we have
$
R_1 = 2(\bar{\gamma}' + \sigma'_1), \  R_s = 2 \Ad_{h S_s}^{-1} \bar{\gamma}',  
$
and
\begin{equation}\label{eq:rec}
R_{i+1} = \Ad_{S_i}R_i + 2\sigma'_{i+1} \quad (i=1, \dots ,s-1).
\end{equation}
Extend the definition of $R_i$ to all $i \in \IZ$ by
$
R_{is+j} = \Ad_h^{-i} R_j$ for  $j=1,2, \dots ,s,\ i \in \IZ$. 
Also let $S_{is+j} = h^{-i} S_j h^i \in \ISto_{is+j}$.
Then we have
$$
R_0 = \Ad_h R_s = 2\Ad_{S_s} \bar{\gamma}', \qquad S_0 = h S_s h^{-1}, \qquad
R_{s+1} = \Ad_h^{-1} R_1 = 2\Ad_h^{-1}\bar{\gamma}' + 2 \sigma'_{s+1}.
$$
It follows that the recursion relation \eqref{eq:rec} holds for all $i \in \IZ$.

\begin{lem}
Each $R_i$ takes values in $\lh \oplus \ls_i$.
\end{lem}

\begin{proof}
Fix $i \in \IZ$ and take arbitrary $j \in \IZ$ with $j<i \leq j+l$.
Since the subgroups $\ISto_{j+1}, \ISto_{j+2}, \dots , \ISto_{j+l}$ 
directly span $U_j$, 
the product map
\[
\ISto_{j+l} \times \cdots \times \ISto_{j+1} \to U_j
\]
is an isomorphism of varieties. 
In particular, the differential of the above map at 
the point $(S_{j+l}, \dots , S_{j+1})$ is surjective, so we have
\begin{align*}
\Lie (U_j) &= (L_{[j+l,j+1]})_*^{-1}(T_{[j+l,j+1]} U_j) \\
&= (L_{[j+l,j+1]})_*^{-1} \sum_{m=j+1}^{j+l} (L_{[j+l,m+1]})_*(R_{[m-1,j+1]})_*(T_{S_m} \ISto_m) \\
&= \sum_{m=j+1}^{j+l} \Ad_{[m-1,j+1]}^{-1} \ls_m,
\end{align*}
where $L, R$ denotes the left and right translations, respectively.
Hence
\begin{equation}\label{eq:span}
\Lie (U_j) = \Ad_{[i-1,j+1]} \Lie (U_j) = \sum_{m=j+1}^{i-1}\Ad_{[i-1,m]}\ls_m 
+ \ls_i + \sum_{m=i+1}^{j+l} \Ad_{[m-1,i]}^{-1} \ls_m.
\end{equation}
On the other hand, using the recursion relation \eqref{eq:rec} repeatedly 
we obtain
\[
R_i = \Ad_{[i-1,m]} R_m + 2 \sum_{q=m+1}^i \Ad_{[i-1,q]} \sigma'_q 
\quad (m<i).
\]
If $m \geq j+1$, each $\Ad_{[i-1,q]} \sigma'_q$ appearing on the right hand side 
is orthogonal to $\Ad_{[i-1,m]} \ls_m$ thanks to \eqref{eq:span}. 
Since $R_m \in \ls_m^\perp$, we see that 
$R_i$ is orthogonal to $\Ad_{[i-1,m]} \ls_m$ for $j<m<i$.
Similarly, using the formula
\[
R_i = \Ad_{[m-1,i]}^{-1} R_m - 2 \sum_{q=i+1}^m \Ad_{[q-1,i]}^{-1} \sigma'_q 
\quad (m>i)
\]
deduced from \eqref{eq:rec}, 
we see that $R_i$ is orthogonal to $\Ad_{[m-1,i]}^{-1} \ls_m$ for $i<m \leq j+l$.
Thus we obtain
\[
R_i \in \left( \sum_{m=j+1}^{i-1}\Ad_{[i-1,m]}\ls_m 
+ \ls_i + \sum_{m=i+1}^{j+l} \Ad_{[m-1,i]}^{-1} \ls_m
\right)^\perp = \Lie (U_j)^\perp,
\] 
and hence
\[
R_i \in \bigcap_{j=i-l}^{i-1} \Lie (U_j)^\perp 
= \bigcap_{m=i-l+1}^{i+l-1} \ls_m^\perp = \lh \oplus \ls_i.
\]
\end{proof}
For $i>0$ we have
$
\pi_\lh(R_i) = \pi_\lh(\Ad_{S_{i-1}}R_{i-1}) = \pi_\lh(R_{i-1}) = \cdots =
\pi_\lh(R_1) = 2 \pi_\lh(\bar{\gamma}')=0,
$
and hence $R_i \in \ls_i$ for all $i \in \IZ$.
Then \eqref{eq:rec} implies
$R_{i+1} = 2 \sigma'_{i+1}, \  R_i =0 \  (i \in \IZ),$
since $\ls_i \cap \ls_{i+1}=0$.
This shows $R_i = \sigma'_i =0$ for all $i \in \IZ$ and $\bar{\gamma}' = 0$.
\epf

Thus Thm. \ref{thm: qh disc} holds.
In turn this implies the global result
Thm. \ref{thm: main} (i.e. Thm. \ref{thm: thm1} in the introduction)
as in the proof of \cite{gbs} Thm. 8.2 (but allowing twists).
In brief  %
$$\THom_\IS(\Pi,G) \cong 
\cA(Q_1)\fus\cdots \fus\cA(Q_m)\fus 
\ID_1\fus\cdots \fus \ID_g\spq G$$
where each $\ID_i$ is a twisted internally fused double, $g$ is the genus of $\wh \Si$ and 
$m$ is the number of boundary components of $\wh \Si$.
Note that many of the other results of \cite{gbs} now extend to the twisted case with little further effort.

The spaces $\cA(Q)$ with just one level will be called 
``twisted fission spaces''. %
As in \cite{fission} and \cite{gbs} \S3.2
they lead to many other Poisson or symplectic varieties beyond
the twisted wild character varieties.
We will call  the class of such varieties``twisted fission varieties''.
In general
one obtains spaces of twisted local systems on surfaces with varying structure groups in various components, and fission gives a way to splice 
surfaces in quite complicated ways 
(see the figures in \cite{fission} in the untwisted case).

Note that Li-Bland--Severa \cite{LBS15} recently suggested a different approach to some of the results of \cite{saqh, fission, gbs}, although not incorporating the new braiding.
We expect their approach also works in the twisted case.

\subsection{Example}

There are an abundance of examples; 
for example coming from any system of linear differential equations on any smooth algebraic curve. 
More specifically recall from Prop. \ref{prop: gln irclass}
that any map $\pi_0(\cI)\to \IN$ (mapping all but a finite number of components to zero) 
determines an irregular class 
for a general linear group with trivial twist $\phi$.
In particular for any positive $c\in \IQ$ 
we could consider the irregular class $Q$ determined by
$I=\langle z^{-c}\rangle\subset \cI$, 
with any multiplicity $n>0$.
Taking $c=k/2$ for any odd integer $k$ 
gives the following twisted quasi-Hamiltonian spaces.

Let $W$ be a complex vector space of dimension $n$ and let $V=W\oplus W$.
Thus elements of $G:=\GL(V)$ can be written as $2\times2$  block matrices, and we consider the following subgroups of $G$:
$$ 
U_+ = \bmx 1 & * \\ 0 & 1 \emx,\qquad 
U_- = \bmx 1 & 0 \\ * & 1 \emx,\qquad
H = \bmx * & 0 \\ 0 & * \emx
$$
and the subset 
$H(\partial) = \bsmx 0 & * \\ * & 0 \esmx\subset G$, which is a twist of $H$.
Let 
$$U_\pm^{(k)}  = U_+\times U_- \times U_+\times\cdots$$
where there are $k$ groups in total on the right.
A simple case of Theorem \ref{thm: thm1}
says that the space
$$\papk{G}{H}{c}  =\cA(Q) =  G\times H(\partial)\times U_\pm^{(k)}$$
is a twisted quasi-Hamiltonian 
$G\times H$ space with moment map
$$\mu=(\mu_G,\mu_H):{\papk{G}{H}{c}}\to G\times H(\partial);$$
$$
\mu_G(C,\bS,h) = C^{-1}hS_k\cdots S_2S_1C,\qquad \mu_H(C,\bS,h) = h^{-1},$$
where $C\in G, h\in H(\partial)$, and
$\bS=(S_1,\ldots,S_k)\in U_\pm^{(k)}$.

These examples are related to the matrix Painlev\'e 1 hierarchy, 
and will be studied in great detail in \cite{by-ccc}.
They are related to the odd Euler continuants,
similarly to how the Painlev\'e 2 hierarchy is 
 related to the even continuants \cite{even-euler}.

The reader can readily write down more examples.
In the examples $I=\langle z^{-c}\rangle$, with $n=1$, 
the corresponding underlying twisted wild character varieties
have been (conjecturally) related to the HOMFLY polynomial of the corresponding torus knot 
(whose projection is the Stokes diagram) in \cite{STZ}.

\appendix

\section{}\label{apx: untwisting}
To make the Stokes groups really explicit in the twisted case 
we will see they are the same as those occurring in the untwisted case, as follows.
In brief we can pull back to a ramified covering disc to reduce to the untwisted case 
(see also \cite{L-R94} \S II.4).
Suppose we are in the situation of \S\ref{ssn: discs}.
Choose an integer $r\ge 1$ and another disc $\De'$ with coordinate $w$, so that the map 
$$\pi : \De' \to \De;\qquad w\mapsto z=w^r$$
is well-defined. Let $\pi$ also denote the induced map 
$\wh \De'\to \wh \De$ on the blowups.
Let $\partial' = \pi^{-1}(\partial)$  and let 
$\wt\De' = \pi^{-1}(\wt \De)$.
Write $b=b_0$ and choose a basepoint $b'\in \partial'$ lying over $b$.
Then all the data at $b$ lifts to data at $b'$.
In particular each element $q$ of the character lattice 
$I_b=X^*(\IT_b)\subset \cI_b$
lifts uniquely to a local section $q'$ (near $b'$)
of the exponential local system $\cI'$ upstairs.
(In brief $q'(w)=q(w^r)$ as functions defined near the basepoints.)
Now choose $r$ so that for any $q\in I_b$ the lift 
$q'$ extends to a global section 
of $\cI'$ over $\partial'$. (For example take $r$ to be the lowest common multiple of the ramification degrees of a $\IZ$-basis of $I_b$.) 
Thus the local system $I\to \partial$ lifts to a {\em constant} lattice 
$I'\subset \cI'$.
Thus $I'$ lies in the unramified part
$\IC\flp w \frp/\IC\flb w \frb$ of $\cI'$.
Moreover we have an isomorphism $I'\cong I_b$ (of lattices).
Thus if we choose a maximal torus $T\subset G$ such that $\IT_b\subset T$, then we have a map
$$X^*(T) \onto I_b\into \IC\flp w \frp/\IC\flb w \frb.$$
Since $\lt^*=X^*(T)\otimes_\IZ \IC$, such a map is just an element
$Q=\sum_1^N {A_i}/{w^i}\in \lt\flp w \frp/\lt\flb w \frb$
where $A_i\in \lt=\Lie(T)$,  i.e. it is an untwisted irregular type, as studied in \cite{gbs}.
In particular the group $H\subset G$ here equals that in \cite{gbs}.
We can then use the (more concrete) definitions of 
singular directions $\IA'$ and Stokes groups 
from \cite{gbs}, and find:

1) $\IA'=\pi^{-1}(\IA)\subset \partial'$

2) If $d\in \IA'$ is on the first sheet of $\De'$ then the Stokes groups are equal:
$\ISto_{\pi(d)} = \ISto'_d\subset G,$
where  $\ISto'_d\subset G$ is the Stokes group at $d$ as defined in \cite{gbs}, and
$\ISto_{\pi(d)} \subset G$ is as defined here (translated to $G=\Aut(\cF_b)$
as in \S\ref{sn: classn}). 
Here the ``first sheet'' is the subset of $\De'$ with argument
$\arg(w)$ satisfying $\arg(b')\ \le\  \arg(w) \ <\  \arg(b')+2\pi/r.$

\renewcommand{\baselinestretch}{1}              %
\normalsize
\bibliographystyle{amsplain}    \label{biby}
\bibliography{../thesis/syr}

\ 

\noindent
D\'epartement de Math\'ematiques \\
B\^{a}timent 425 \\
Universit\'e Paris-Sud \\
91405 Orsay \\
France

\noindent
philip.boalch@math.u-psud.fr

\vspace{0.5cm} 
\noindent
Department of Mathematics \\ 
Tokyo Institute of Technology \\
Tokyo 152-8551 \\
Japan 

\noindent
yamakawa@math.titech.ac.jp

\end{document}